\documentclass[]{article}
\usepackage{amssymb,amsfonts,amscd,amsmath,amsthm}

\def\F{{\mathcal F}}

\def\B{{\mathbb B}}
\def\G{{\mathcal G}}
\def\H{{\mathcal H}}
\def\Z{{\mathbb Z}}
\def\N{{\mathbb N}}

\newtheorem{theorem}{Theorem}
\newtheorem{proposition}{Proposition}
\newtheorem{definition}{Definition}
\newtheorem{lemma}{Lemma}
\newtheorem{remark}{Remark}
\newtheorem{corollary}{Corollary}
\newtheorem{example}{Example}

\title{Unique maximal Betti diagrams for Artinian Gorenstein
    $k$-algebras with the weak Lefschetz property}

\author{Ben Richert}

\begin{document}
\maketitle

\begin{abstract}
  We give an alternate proof for a theorem of Migliore and Nagel. In
  particular, we show that if $\H$ is an SI-sequence, then the
  collection of Betti diagrams for all Artinian Gorenstein $k$-algebras with
  the weak Lefschetz property and Hilbert function $\H$ has a unique
  largest element.
\end{abstract}

\section{Introduction and Background}\label{s:introduction}
The graded Betti numbers of a module have received quite a bit of
attention, especially since the advent of computer algebra systems
allowing for the computation of examples in bulk. In this paper we
explore the relationship between graded Betti numbers and the weaker
Hilbert function invariant. In particular, we will take up a piece of
the Gorenstein version of the question, {\em given a Hilbert function
  $\H$, what graded Betti numbers can occur?}

To fix notation, we let $k$ be an infinite field. Then given a
polynomial ring $R$ and a homogeneous ideal $I\subseteq R$, we write
$H(R/I):\N\to \N$ to be the Hilbert function of $R/I$, so
$H(R/I,d)=\dim_k (R/I)_d$, and
$\beta^I_{i,j}=\dim_k({\rm Tor}_i(R/I,k))_j$ to be the
$(i,j)^{\rm th}$ graded Betti number of $R/I$. We will write
$\beta^{I}$ to refer to the Betti diagram of $R/I$ (which, following
the notation of the computer algebra system Macaulay 2, is a table
whose $(i,j)^{\rm th}$ entry is $\beta^I_{i,i+j}$) as a convenient way
to refer to all the graded Betti numbers of a module at once. Given an
$\mathcal O$-sequence $\H$ (that is, given any valid Hilbert function)
with $\H(0)=1$ we write ${\B}_{\H}$ to be the set of all $\beta^I$
such that $I\subseteq R$ is homogeneous with $H(R/I)=\H$.

To understand ${\B}_{\H}$ for fixed $\H$, the first step is to show
that this set has a sharp upper bound under the obvious component-wise
partial order. This was accomplished independently by Bigatti
\cite{Bigatti} and Hulett \cite{Hulett} in characteristic zero, and
later by Pardue \cite{Pardue} in characteristic $p$. These authors
showed that if $L$ is the lex ideal (guaranteed by Macaulay to exist)
such that $H(R/L)=\H$, then $\beta^L$ is the unique largest element of
${\B}_{\H}$. Various other authors have asked if important subsets of
$\B_{\H}$ also have unique upper bounds. For example, the Lex Plus
Powers Conjecture of Evans and Charalambous \cite{MerminMurai}
predicts that restricting $\B_{\H}$ to the Betti diagrams of quotients
of ideals containing regular sequences in prescribed degrees should
have a unique max. Aramova, Hibi, and Herzog \cite{AHH:SF} proved that
restricting $\B_{\H}$ to the Betti diagrams of quotients of squarefree
monomial ideals does in fact have a unique largest element.

One obvious subset of $\B_{\H}$ to consider consists of the Betti
diagrams of graded Artinian Gorenstein $k$-algebras. In general we
cannot predict when this subset is nonempty since there is currently
no analog to Macaulay's theorem for Gorensteins, but there is one well
understood class of Hilbert functions, and we turn our attention
there.

\begin{definition}\label{d:SIsequence}
  Given a $\mathcal O$-sequence $\H$, let $\Delta\H$ be the sequence
  with $\Delta\H(d)= \H(d)-\H(d-1)$, and, for $i>1$, let
  $\Delta^i\H=\Delta\Delta^{i-1}(\H)$. 

  When $\H=0$ we say $\ell(\H)=-1$, and for $\H\ne 0$, we write
  $$\ell(\H)=\max\{d\in \N\ |\ \H(d)\ne 0\}$$ if such a $d$ exists and
  $\ell(\H)=\infty$ otherwise.  Finally, given an
  $\mathcal O$-sequence $\H$ with $0<\ell(\H)<\infty$, we say that
  $\H$ is an SI-sequence (for Stanley-Iarrobino) if
  $\H(d)=\H(\ell(\H)-d)$ for all $0\le d\le \ell(\H)$ (that is, $\H$
  is symmetric), and $\Delta\H$ is an $\mathcal O$-sequence for
  $d\le\lfloor \frac{\ell(H)}{2}\rfloor$ (that is, the first half of
  $\H$ is a differentiable $\mathcal O$-sequence).
\end{definition}

It is a result of Harima \cite{Harima} that given an SI-sequence $\H$,
there is a Artinian Gorenstein $k$-algebra with Hilbert function $\H$. In fact,
more is true.

\begin{definition}\label{d:WLP}
  Let $M$ be a finite length graded $R$-module, with Hilbert function
  $\H$. Then $M$ is said to have the weak Lefschetz property if
  $\Delta\H(i)\ge 0$ for all $i\le \lfloor\frac{\ell(\H)}{2}\rfloor$,
  $\Delta\H(i)\le 0$ for all $i> \lfloor\frac{\ell(\H)}{2}\rfloor$
  (that is, $\H$ is unimodal), and there is $r\in R_1$ such that
  $M_i\to M_{i+1}$ given by multiplication by $r$ is injective or
  surjective for all $i = 0, \dots, \ell(\H)$.
\end{definition}

Harima actually proved that $\H$ is an SI-sequence if and only if
there is a Artinian Gorenstein $k$-algebra with the weak Lefschetz property and
Hilbert function $\H$.

\begin{remark}
  Given an SI-sequence $\H$, we will abuse notation slightly by saying
  that an $R$-module $M$ has Hilbert function $\Delta H$ when we
  really mean that its Hilbert function equals $\Delta H$ for
  $d\le \lfloor\frac{\ell(\H)}{2}\rfloor$, and is zero otherwise.
\end{remark}

The Gorenstein question was first considered in this context by
Geramita, Harima, and Shin \cite{GHS:AM} in 2000. They demonstrated
how to embed a standard $k$-configuration $\mathbb X$ (an iteratively
defined sets of points in ${\mathbb P}^n$) in a so-called basic
configuration ${\mathbb Z}$ (an intersection of unions of
hyperplanes), and then showed that the quotient of the sum of the
defining ideals of ${\mathbb X}$ and ${\mathbb Z}-{\mathbb X}$ is
Gorenstein with the weak Lefschetz property and that its Betti diagram
is larger than that of any other Artinian Gorenstein $k$-algebra with the weak
Lefschetz property and the same Hilbert function. Of course, this only
settled the question (granting the weak Lefschetz condition) for
Hilbert functions which can be obtained via the given construction,
and one can show that this does not include every possible
SI-sequence.

In 2003, Migliore and Nagel \cite{MiglioreNagel:Gorensteins} extended
Geramita, Harima, and Shin's result by showing that for {\it any}
SI-sequence $\H$, the collection of Betti diagrams for Artinian
Gorenstein $k$-algebras with the weak Lefschetz property and Hilbert
function $\H$ has a unique largest element. This is done in two
steps---giving an upper bound for the Betti diagrams in question, and
showing that this upper bound is sharp.  Establishing the upper bound,
which we record here for use later, turns out to be the easier step.

\begin{theorem}[Theorem 8.13 in
  \cite{MiglioreNagel:Gorensteins}]\label{t:MiglioreNagel}
  Let $S=k[\mu_1, \dots, \mu_c]$ for $k$ a field, $\H$ be an
  SI-sequence with $\H(1)\ge 1$, $c=\H(1)$, $t=\ell(\Delta\H)$, $L$ be
  the lex ideal in $S/(x_\mu)$ with Hilbert function $\Delta \H$, and
  $I\subseteq S$ be a homogeneous ideal such that $S/I$ is a
  Artinian Gorenstein $k$-algebra with the weak Lefschetz property and
  $H(S/I)=\H$. Then $$\beta^I_{i,i+j}\le
  \begin{cases} \beta^L_{i,i+j} & \mbox{if }j\le \ell(\H)-t-1\\
    \beta^L_{i,i+j} +\beta^L_{c-i,c-i+\ell(\H)-j} & \mbox{if }
    \ell(\H)-t\le j \le t\\ \beta^L_{c-i,c-i+\ell(\H)-j} & \mbox{if }
    j \ge t+1.\\ \end{cases}$$
\end{theorem}

Giving the result with respect to $(i,i+j)$ makes it straight forward
to interpret the bound in terms of Betti diagrams. In particular,
Migliore and Nagel's theorem says that taking two copies of the Betti
diagram of the lex ideal with Hilbert function $\Delta \H$, rotating
the second by 180 degrees and degree shifting appropriately, then
super-imposing these tables and adding entries gives an upper bound
for the Betti diagram of any Artinian Gorenstein $k$-algebra with the weak
Lefschetz property and Hilbert function $\H$.

Demonstrating that this bound is sharp turned out to be more
difficult.  Beginning with an arbitrary SI-sequence $\H$, the authors
define a special generalized stick figure (that is, a union of linear
subvarieties of ${\mathbb P}^n$ of the same dimension $d$ such that
the intersection of any three components has dimension $\le d-2$)
whose Hilbert function is $\Delta\H$, and embed this set in a second
generalized stick figure with a Gorenstein coordinate ring and whose
Hilbert function has a certain maximal property. The sum of the
defining ideal of the original space with its link in the manufactured
Gorenstein ideal gives a Gorenstein quotient with the weak Lefschetz
property, the correct Hilbert function, and maximal graded Betti
numbers.

In this paper, we give a new proof that Migliore and Nagel's bound is
sharp. We do this by way of monomial ideals and a doubly iterative
procedure, making the description more compact and the overall proof
shorter and more naive (in the sense that it relies mostly on double
induction). Furthermore, because the construction is monomial until
the last possible moment, it is easy to actually compute the ideals in
question, for instance on the computer algebra system Macaulay 2, even
for `large' $\H$. In fact, it was mostly an attempt to 
compute examples of Migliore and Nagel's ideals which led to this
alternative proof. Admittedly, the economy and computability obtained
by our approach comes at a steep cost, because the important geometric
intuition and insight inherent in the generalized stick figures used
in Migliore and Nagel's work is lost in our machinery.

By section, our procedure will be as follows.  In \S
\ref{s:buildingblock}, we introduce the main building block of our
construction---given a Hilbert function $\H$ with $\ell(\H)<\infty$,
$c\ge \H(1)$, and $t\ge \ell(\H)$, we iteratively define a squarefree
monomial ideal $I_{c,t}(\H)$ via decomposition of $\H$ into two
${\mathcal O}$-sequences.  In \S \ref{s:properties}, we show that the
quotient of $I_{c,t}(\H)$ (in the appropriate polynomial ring) is
Cohen-Macaulay, and compute its dimension, Hilbert function, and
graded Betti numbers.  In \S \ref{s:subideals}, we consider a special
case of our procedure for manufacturing the $I_{c,t}(\H)$, and thereby
construct a family of Gorenstein ideals
$G_{c,t,s}\subseteq I_{c,t}(\H)$ for $s\ge 0$.  Finally in \S
\ref{s:gorenstein}, given an SI-sequence $\H$ we form a Gorenstein
ideal $J_c(\H)$ that has an Artinian reduction with the weak Lefschetz
property, Hilbert function $\H$, and extremal graded Betti numbers.
For $\H(1)\ge 2$, the procedure is to let $c=\H(1)$,
$t=\ell(\Delta\H)$, $s=\ell(\H)-2t+1$, and then sum
$I_{c-1,t}(\Delta\H)$ and its link with respect to $G_{c-1,t,s}$.

\section{The main building block}\label{s:buildingblock}
At the heart of our construction is the decomposition of a Hilbert
function into two $\mathcal O$-sequences as follows.

\begin{definition}\label{d:HF}
  Given an $\mathcal O$-sequence $\H$, let $c\ge \max\{1,\H(1)\}$,
  $T=k[\mu_1, \dots , \mu_c]$ for $k$ a field, and $L\subseteq T$ be
  the lex ideal attaining $\H$. We define $\H_{c}$ to be the
  $\mathcal O$-sequence $\H_c=H(T/(L+\mu_1))$, and $\H^c$ to be the
  $\mathcal O$-sequence $\H^c=H(T/(L:\mu_1))$.
\end{definition}

Given two $\mathcal O$-sequences $\H'$ and $\H$, if we write $\H'\le
\H$ to indicate that $\H'(d)\le \H(d)$ for all $d\in \N$, then it is
easy to see that $\H'\le\H$ implies $\H'_{c}\le \H_{c}$ and
$\H'^{c}\le
\H^{c}$. 
A simple definition chase demonstrates that $(\H^c)_{c}\le
(\H_{c})^{c-1}$ (where the latter term makes sense since $\H_c(1)\le
c-1$).

\begin{lemma}\label{l:HFLemma}
  Suppose that $c\ge 2$ 
  and $\H$ is an $\mathcal O$-sequence with $H(1)\le c$.  Then
  $(\H^c)_{c}\le (\H_{c})^{c-1}$.
\end{lemma}
\begin{proof}
  Let $L$ be the lex ideal in $T=k[\mu_1, \dots, \mu_c]$ attaining
  $\H$. It is easy to see that $(L:\mu_1)$ is lex, so
  that $$(\H^{c})_{c}= (H(T/(L:\mu_1)))_{c} =
  H(T/((L:\mu_1)+\mu_1)).$$ On the other hand, $T/(L+\mu_1) \cong
  T'/L'$ for $T'=k[\lambda_1, \dots , \lambda_{c-1}]$ and $L'$ the
  (obviously lex) preimage of $L\cap k[\mu_2, \dots, \mu_c]$ under the
  map $\phi:T'\to T$ by $\lambda_i\to \mu_{i+1}$.
  So $$(\H_{c})^{c-1}= (H(T/L+\mu_1))^{c-1} = H(T'/(L':\lambda_1))$$
  and it is enough to show that there is a surjective homomorphism
  $T'/(L':\lambda_1)\to T/((L:\mu_1)+\mu_1)$. Obviously $\phi$ induces
  a surjection from $T'$ to $T/((L:\mu_1)+\mu_1)$, so it is enough to
  show that if $m$ is a monomial in $(L':\lambda_1)$, then $\phi(m)\in
  (L:\mu_1)$. But $m\lambda_1\in L' \implies
  \phi(m\lambda_1)=\phi(m)\mu_2\in L\cap k[\mu_2, \dots ,
  \mu_c]\subseteq L$, so $\phi(m)\mu_1\in L$ (since $L$ is lex) and
  hence $\phi(m)\in (L:\mu_1)$.
\end{proof}

From the short exact sequence
$$0\to \frac{T}{L:\mu_1}(-1) \xrightarrow{\mu_1} \frac{T}{L}\to
\frac{T}{L+\mu_1}\to 0$$
it is evident that $\H(d)=\H_c(d)+\H^c(d-1)$ for all $d\in \N$, and we
have already observed that $\H_c(1)\le c-1$. These observations
provide the minor justification required to define the central object
in our construction.

\begin{definition}\label{d:IctH}
  For $c,s\in \N_{\ge 0}$ and $t\in \Z_{\ge -1}$, let
 $$R_{c,t,s}=k[x_0,
 \dots, x_{t+\lfloor\frac{c-1}{2}\rfloor}, y_0, \dots,
 y_{t+\lfloor\frac{c-2}{2}\rfloor}, z_0, \dots, z_{s-1}].$$

 \noindent
 Then given an $\mathcal O$-sequence $\H$ with $\ell(\H)<\infty$,
 $c\ge \H(1)$, and $t\ge \ell(\H)$ we define the ideal
 $I_{c,t}(\H)\subseteq R_{c,t,0}$ iteratively as follows: \newline

\noindent
if $\H=0$, we let $I_{c,t}(\H)= R_{c,t,0}$ (and thus, if $t=-1$,
$I_{c,-1}(\H)=R_{c,-1,0}$ since $t\ge \ell(\H)$); \newline

\noindent if $\H\ne 0$ and $c=0$ we let $I_{0,t}(\H)=0\subseteq R_{0,t,0}$;
\newline

\noindent if $\H\ne 0$, $c>0$, and $t>-1$, we define
  $$I_{c,t}(\H)=I_{c-1,t}(\H_{c})R_{c,t,0}+\omega_{c,t}I_{c,t-1}
  (\H^{c})R_{c,t,0},$$ where
  $\omega_{c,t}=x_{t+\lfloor\frac{c-1}{2}\rfloor}$ if $c$ is odd and
  $y_{t+\lfloor\frac{c-2}{2}\rfloor}$ if $c$ is even.
 \end{definition}

\begin{remark} 
  We will suppress the obvious natural inclusions $$R_{c-1,t,s},
  R_{c,t-1,s}, R_{c,t,s-1}\subseteq R_{c,t,s},$$ and hence will
  write $$I_{c,t}(\H)=I_{c-1,t}(\H_{c})+
  \omega_{c,t}I_{c,t-1}(\H^{c})$$ in cases for which there can be no
  confusion. Whenever possible, we will write $R$ for $R_{c,t,s}$, for
  example $R/I_{c,t}(\H)$ for $R_{c,t,s}/I_{c,t}(\H)$. Since $s=0$ for
  most of the initial construction, we will will usually write
  $R_{c,t}$ for $R_{c,t,0}$ when more precise information about the
  ambient ring is required.
\end{remark}

%
%

\begin{example}\label{e:IctH}
  Consider the Hilbert function $\H=\{1,2\}$. Then
  \begin{eqnarray*}
    I_{3,2}(\{1,2\})&=& I_{2,2}(\{1,2\})+\omega_{3,2} I_{3,1}(\{0\}) = (x_1x_2,x_1y_2,y_1y_2,x_3)\\
    \\
    I_{2,2}(\{1,2\})&=& I_{1,2}(\{1,1\})+\omega_{2,2} I_{2,1}(\{1\}) =
    (x_1x_2, x_1y_2,y_1y_2)\\
    I_{3,1}(\{0\})&=& R_{3,1} \\
    \\
    I_{1,2}(\{1,1\})&=& I_{0,2}(\{1\})+\omega_{1,2} I_{1,1}(\{1\}) = (x_1x_2)\\
    I_{2,1}(\{1\})&=& I_{1,1}(\{1\})+\omega_{2,1} I_{2,0}(\{0\}) = (x_1,y_1)\\
    \\
    I_{0,2}(\{1\})&=& 0\\
    I_{1,1}(\{1\})&=&  I_{0,1}(\{1\})+\omega_{1,1} I_{1,0}(\{0\}) = (x_1)\\
    I_{2,0}(\{0\})&=& R_{2,0}\\
    \\
    I_{0,1}(\{1\})&=& 0\\
    I_{1,0}(\{0\})&=&  R_{1,0}\\
\end{eqnarray*}
One can easily show that for $\H\ne 0$ with $\H(1)\le 1$ and
$t\ge \ell(\H)$, $I_{1,t}(\H)=(x_0\cdots x_t)$.
\end{example}

The $I_{c,t}(\H)$ are relatively easy to manipulate inductively. To
illustrate this (and for use later), we record here the following
observations.

\begin{proposition}\label{p:sqfree}
  Let $\H$ be an $\mathcal O$-sequence with $\ell(\H)<\infty$,
  $c\ge \H(1)$, and $t\ge \ell(\H)$. Then $I_{c,t}(\H)$ is a
  squarefree monomial ideal.
\end{proposition}
\begin{proof}
  The result is obvious if $\H=0$, $c=0$, or $t=-1$,
  so suppose that $\H\ne 0$ and $c>0$, $t>-1$ and let $m\in
  I_{c,t}(\H)$ be a minimal generator. But $m\in
  I_{c-1,t}(\H_{c})+\omega_{c,t} I_{c,t-1}(\H^{c})$ and each of
  $I_{c-1,t}(\H_{c})$ and $I_{c,t-1}(\H^{c})$ are squarefree by the
  induction hypothesis. The result then follows because
  $I_{c,t-1}\subset R_{c,t-1}$ but $\omega_{c,t}\not\in R_{c,t-1}$ by
  construction.
\end{proof}

\begin{proposition}\label{p:IctHlinears}
  Let $\H$ be an $\mathcal O$-sequence with $-1<\ell(\H)<\infty$,
  $c\ge \H(1)$, and $t\ge \ell(\H)$.  Then $I_{c,t}(\H)$ contains
  exactly $c-\H(1)$ minimal linear generators.
\end{proposition}
\begin{proof}
  This is obvious if $c=0$. If $c>1$, then there are two
  possibilities. If $\H^c\ne 0$, then (using the notation from
  Definition (\ref{d:HF})) $\mu_1\not\in L$ so $\H_c(1)=\H(1)-1$. By
  induction $I_{c-1,t}(\H_c)$ contains exactly $c-1-\H_c(1)=c-\H(1)$
  linear generators and $\omega_{c,t}I_{c,t-1}(\H^c)$ is generated in
  degrees $\ge 2$ (the latter point follows from the definition
  because $I_{c,t-1}(\H^c)= R_{c,t-1}$ if and only if $\H^c=0$).
  Since $I_{c,t}(\H)=I_{c-1,t}(\H_c)+\omega_{c,t}I_{c,t-1}(\H^c)$, we
  are finished. If $\H^c=0$, then $\mu_1\in L$, so $\H_c(1)=\H(1)$. By
  induction $I_{c-1,t}(\H_c)$ contains exactly $c-1-\H_c(1)=c-1-\H(1)$
  linear generators, so
  $I_{c,t}(\H)=I_{c-1,t}(\H_c)+\omega_{c,t}R_{c,t}$ contains exactly
  $c-\H(1)$ linear generators as required.
\end{proof}
 
\begin{proposition}\label{p:HFcomparison}
  Let $\H'\le \H$ be $\mathcal O$-sequences with $\ell(\H)<\infty$,
  $c\ge \H(1)$, and $t\ge \ell(\H)$.  Then
  $I_{c,t}(\H')\supseteq I_{c,t}(\H)$.
\end{proposition}
\begin{proof}
  We do induction on $c$ and $t$. The result if obvious if $t=-1$,
  $\H=0$, or $c=0$, so we suppose $c>0$, $t>-1$, and $\H\ne 0$.  Then
  $$I_{c,t}(\H')=I_{c-1,t}(\H'_{c})
  +\omega_{c,t}I_{c,t-1}(\H'^{c})$$
  $$\supseteq I_{c-1,t}(\H_{c}) + \omega_{c,t} I_{c,t-1}(\H^{c}) =
  I_{c,t}(\H)$$
  applying the induction hypothesis as well as the observations
  preceding Lemma (\ref{l:HFLemma}).
\end{proof}

It follows directly from the definition that
$I_{c-1,t}(\H_c)\subseteq I_{c,t}(\H)$. In fact, more is true.

\begin{proposition}\label{p:subsetcIversion}
  Let $\H$ be an $\mathcal O$-sequence, $c\ge \max\{1, \H(1)\}$, and
  $t\ge \max\{0,\ell(\H)\}$, then
  $I_{c-1,t}(\H_{c})R_{c,t}\subseteq I_{c,t-1}(\H^{c})R_{c,t}$.
\end{proposition}
\begin{proof} The result is obvious if $\H=0$, so presume not and
  proceed by induction on $c$.  If $c=1$ then $\H_1\ne 0$ and
  $I_{0,t}(\H_{1})R_{1,t} = (0)\subseteq I_{1,t-1}(\H^{1})R_{1,t}$
  trivially.

  Now we suppose that $c>1$. So
  $$I_{c-1,t}(\H_{c})R_{c,t}=
  I_{c-2,t}((\H_{c})_{c-1})R_{c,t}+\omega_{c-1,t}
  I_{c-1,t-1}((\H_{c})^{c-1})R_{c,t}$$
$$\subseteq I_{c-1,t-1}((\H_{c})^{c-1})R_{c,t}+\omega_{c-1,t}I_{c-1,t-1}
((\H_{c})^{c-1})R_{c,t}$$
$$\subseteq I_{c-1,t-1}((\H_{c})^{c-1})R_{c,t} \subseteq
I_{c-1,t-1}((\H^{c})_{c})R_{c,t}$$
$$ \subseteq I_{c,t-1}(\H^{c})R_{c,t},$$
as required where the first containment is by the induction hypothesis
and the third is by Lemma (\ref{l:HFLemma}) and Proposition
(\ref{p:HFcomparison}).
\end{proof}

The final result of this section gives some indication how we will
chase information about Hilbert functions and graded Betti numbers
through our iterative definition.

\begin{proposition}\label{p:intersection3}
  Let $\H$ be an $\mathcal O$-sequence with $\H(1)\ge 1$ and
  $\ell(\H)<\infty$, $c\ge \H(1)$, and
  $t\ge\min\{d\ |\ \H(d)\ge\H(d+1)\}$. Then we have a short exact
  sequence
  $$0\to \frac{R}{\omega_{c,t}I_{c-1,t}(\H_c)} \to
  \frac{R}{I_{c-1,t}(\H_c)}\oplus
  \frac{R}{\omega_{c,t}I_{c,t-1}(\H^c)} \to \frac{R}{I_{c,t}(\H)}\to
  0.$$
\end{proposition}
\begin{proof}
  That
  $I_{c-1,t}(\H_c)\cap \omega_{c,t}I_{c,t-1}(\H^c)\supseteq
  \omega_{c,t}I_{c-1,t}(\H_c)$
  follows immediately from Proposition (\ref{p:subsetcIversion}), so
  we suppose that $m$ is a monomial in
  $I_{c-1,t}(\H_c)\cap \omega_{c,t} I_{c,t-1}(\H^c)$. Then
  $m\in I_{c-1,t}(\H_c)$ and $\omega_{c,t}\ |\ m$. But $\omega_{c,t}$
  is a non-zero-divisor on $R/I_{c-1,t}(\H_c)$ because
  $\omega_{c,t}\not\in R_{c-1,t}$, so
  $m\in \omega_{c,t}I_{c-1,t}(\H_c)$ as required.

%
\end{proof}

\section{Properties of the $I_{c,t}(\H)$}\label{s:properties}

We now show that $R/I_{c,t}(\H)$ is Cohen-Macaulay and compute its
dimension, graded Betti numbers, and Hilbert function.

\begin{proposition}\label{p:minprimes}
  Let $\H$ be an $\mathcal O$-sequence with $-1<\ell(\H)<\infty$,
  $c\ge \H(1)$, and $t\ge \ell(\H)$. Then any minimal prime of
  $I_{c,t}(\H)$ is generated by $c$ variables, exactly
  $\lfloor\frac{c}{2}\rfloor$ of which are from
  $\{y_0, \dots, y_{t+\lfloor\frac{c-2}{2}\rfloor}\}$.
\end{proposition}
\begin{proof}
  We do induction on $c$ and $t$. If $c=0$ then $I_{0,t}(\H)=(0)$ and
  the result is obvious. If $t=0$, then
  $I_{c,0}(\H)= (x_0, \dots, x_{\lfloor\frac{c-1}{2}\rfloor}, y_0,
  \dots, y_{\lfloor\frac{c-2}{2}\rfloor})$
  and again the result is clear.

  So suppose $c\ge 1$, $t\ge 1$, and $P$ is a minimal prime of
  $$I_{c,t}(\H)=I_{c-1,t}(\H_c)+\omega_{c,t}I_{c,t-1}(\H^c).$$ If
  $\omega_{c,t}\in P$, then the prime obtained by removing
  $\omega_{c,t}$ from $P$ is minimal over $I_{c-1,t}(\H_c)$ and
  hence consists of $c-1$ variables, $\lfloor\frac{c-1}{2}\rfloor$
  from $\{y_0, \dots, y_{t+\lfloor\frac{c-3}{2}\rfloor}\}$. If $c$ is
  even then $\omega_{c,t}=y_{t+\lfloor\frac{c-2}{2}\rfloor}$,
  $\lfloor\frac{c-1}{2}\rfloor = \lfloor\frac{c}{2}\rfloor-1$, and the
  result follows. If $c$ is odd, then
  $w_{c,t}=x_{t+\lfloor\frac{c-1}{2}\rfloor}$ and
  $\lfloor\frac{c-1}{2}\rfloor = \lfloor\frac{c}{2}\rfloor$ as
  required.

  If $\omega_{c,t}\not\in P$ then $\omega_{c,t}\not\in I_{c-1,t}(\H)$,
  so $I_{c,t-1}(\H^c)\ne R_{c,t-1}$ and thus $\H^c\ne 0$. Moreover,
  $I_{c,t-1}(\H^c)\subseteq P$, 
  so there is a prime $Q\subseteq P\cap R_{c,t-1}$ minimal over
  $I_{c,t-1}(\H^c)$. By induction on $t$, $Q$ has $c$ generators of
  which $\lfloor\frac{c}{2}\rfloor$ are in
  $\{y_0, \dots, y_{t-1+\lfloor\frac{c-2}{2}\rfloor}\}$, and thus it
  is enough to show that $QR_{c,t}=P$. But this is immediate since
  $I_{c-1,t}(\H_c)\subseteq I_{c,t-1}(\H^c)R_{c,t}$ by Proposition
  (\ref{p:subsetcIversion}).
\end{proof}

\begin{corollary}\label{c:minprimes}
  Let $\H$ be an $\mathcal O$-sequence with $-1<\ell(\H)<\infty$,
  $c\ge \H(1)$, $t\ge \ell(\H)$, and let $s\ge 0$. Then
  $\dim R_{c,t,s}/I_{c,t}(\H)=2t+s$.
\end{corollary}

To show that $R/I_{c,t}(\H)$ is Cohen-Macaulay requires a few ring
theoretic observations. First recall the well known fact that for any
graded ring $R$ and $I\subseteq R$ a homogeneous ideal, if $R/I$ is
Cohen-Macaulay, then $R[x]/IR[x]$ is Cohen-Macaulay
 with $\dim (R[x]/IR[x])= \dim(
R/I)+1$.
Moreover by tensoring a minimal free resolution of $R/I$ by $R[x]$,
we observe that ${\rm pd}(R[x]/IR[x])= {\rm pd}(R/I)$, and the graded
Betti numbers do not change. 

A less standard but equally easy fact is as follows.

\begin{lemma}\label{l:multbyomega}
  Let $R$ be a graded ring and $I\subseteq R$ be a homogeneous ideal
  such that $R/I$ is Cohen Macaulay. If $f$ is a $d$-form of $R$ which
  is a non-zero-divisor on $R/I$, then ${\rm depth} (R/fI)= {\rm depth}
  (R/I)$, ${\rm pd} (R/fI)= {\rm pd}(R/I)$, and 
  $\beta^{I}_{i,j}=\beta^{fI}_{i,j+d}$ for all $i,j\in \N$.
\end{lemma}
\begin{proof}
  Let $(f_1, \dots, f_r)$ be a minimal generating set for $I$ with
  $d_i=\deg f_i$ for $i=1, \dots, r$. Then $(ff_1, \dots, ff_r)$ is a
  minimal generating set for $fI$ and setting $\F_{\bullet}$ and
  $\G_{\bullet}$ to be minimal free resolutions of $R/I$ and $R/fI$
  with respective differentials $\delta_i$ and $\partial_i$, we may
  take $$\delta_1= \left[\begin{matrix}f_1\\ \vdots\\
      f_r\end{matrix}\right]\mbox{ and }\partial_1=
  \left[\begin{matrix}ff_1\\ \vdots\\
      ff_r\end{matrix}\right].$$
  Note that $\ker\delta_1$ is a submodule of
  $R(-d_1)\oplus \cdots \oplus R(-d_r)$ while $\ker\partial_1$ is a
  submodule of $R(-d_1-d)\oplus \cdots \oplus R(-d_r-d)$. But
  $[g_1,\dots, g_r]\in \ker\delta_1$ if and only if it is in the
  kernel of $\partial_1$, that is $\sum_{i=0}^r f_ig_i = 0 $ if and
  only if $\sum_{i=0}^r ff_ig_i= f\sum_{i=0}^r f_ig_i =0$, and thus
  $\ker\delta_1(-d)\cong \ker\partial_1$ in the obvious way. Thus the
  assertions concerning resolutions follows immediately and
  ${\rm depth}(R/fI) = {\rm depth}(R/I)$ by the Auslander-Buchsbaum
  formula. 
\end{proof}

We can now show that $R/I_{c,t}(\H)$ is Cohen-Macaulay.

\begin{theorem}\label{t:IisCM}
  Let $\H$ be an $\mathcal O$-sequence with $-1<\ell(\H)$,
  $c\ge \H(1)$, and $t\ge \ell(\H)$.  Then $R_{c,t,s}/I_{c,t}(\H)$ is
  a dimension $2t+s$ Cohen-Macaulay algebra with projective dimension
  $c$.
\end{theorem}
\begin{proof}
  It suffices, by Corollary (\ref{c:minprimes}) and the
  Auslander-Buchsbaum formula, to show that $R_{c,t,s}/I_{c,t}(\H)$
  has projective dimension $c$.

  We proceed by induction on $c$ and $t$. When $c=0$, $I_{0,t}(\H)=0$
  so the result is obvious. If $t=0$ then
  $I_{c,0}(\H)=(x_0, \dots, x_{\lfloor\frac{c-1}{2}\rfloor}, y_0,
  \dots, y_{\lfloor\frac{c-2}{2}\rfloor}),$
  and the Koszul complex givens the result.

  Now suppose $c,t>0$ so that
  $I_{c,t}(\H)= I_{c-1,t}(\H_c)+\omega_{c,t} I_{c,t-1}(\H^c).$


  If $\H^c=0$, then $I_{c,t}(\H)=I_{c-1,t}(\H_c)+\omega_{c,t}$ and the
  mapping cone resolution on the short exact sequence
  $$0\to \frac{R}{I_{c-1,t}(\H_c)} \xrightarrow[]{\omega_{c,t}}
  \frac{R}{I_{c-1,t}(\H_c)} \to \frac{R}{I_{c,t}(\H)}\to 0$$
  is minimal (because $\omega_{c,t}$ is regular modulo
  $I_{c-1,t}(\H_c)$), so the result follows immediately by induction.

  If $\H^c\ne 0$, then consider the sequence, exact by Proposition
  (\ref{p:intersection3}),
  $$0\to \frac{R}{\omega_{c,t}I_{c-1,t}(\H_c)} \to
  \frac{R}{I_{c-1,t}(\H_c)}\oplus
  \frac{R}{\omega_{c,t}I_{c,t-1}(\H^c)} \to \frac{R}{I_{c,t}(\H)}\to
  0.$$
  Let $\F_{\bullet}$, $\G_{\bullet}$, and $\H_{\bullet}$ be minimal
  free resolutions of, respectively, $R/\omega_{c,t} I_{c-1,t}(\H_c)$,
  $R/I_{c-1,t}(\H_c)$, and $R/\omega_{c,t} I_{c,t-1}(\H^c)$, and let
  $T_\bullet$ be the minimal free resolution of $R/I_{c,t}(\H)$ living
  inside the mapping cone resolution. Thus
  $T_r\subseteq F_{r-1}\oplus G_r\oplus H_r=0$ for $r>c$ since
  ${\rm pd}\left(R/\omega_{c,t}I_{c-1,t}(\H_c)\right) = {\rm
    pd}\left(R/I_{c-1,t}(\H_c)\right)=c-1,$
  and ${\rm pd}\left(R/\omega_{c,t}I_{c,t-1}(\H^c)\right)=c$ (by
  induction and Lemma (\ref{l:multbyomega})). Furthermore,
  $0\ne H_c\subseteq T_c$ because $F_c=0$ and hence any non-minimality
  in the mapping cone resolution cannot involve $H_c$.
  Thus the projective dimension of $R/I_{c,t}(\H)$ is $c$ as required.
\end{proof}

Now we show that the Betti diagrams of $I_{c,t}(\H)$ and the lex
ideal attaining $\H$ coincide. The first step is to demonstrate that
the graded Betti numbers of the latter can be decomposed via $\H_c$ and
$\H^c$.

\begin{lemma}\label{l:lexbettinumbers}
  Let $\H$ be an $\mathcal O$-sequence with $-1<\ell(\H)<\infty$,
  $c\ge \max\{1, \H(1)\}$, $L$ be the lex ideal in
  $T=k[\mu_1, \dots, \mu_c]$ with Hilbert function $\H$, and $L'$ be
  the lex ideal in $T'=k[\lambda_1, \dots, \lambda_{c-1}]$ with
  Hilbert function $H(T/(L+\mu_1))$. Then for
  $(0,1)\ne (i,j)\ne (1,1)$,
  $\beta^{L}_{i,j}=\beta^{L+\mu_1}_{i,j}+\beta^{L:\mu_1}_{i,j-1}=
  \beta^{L'}_{i-1,j-1}+\beta^{L'}_{i,j} +\beta^{L:\mu_1}_{i,j-1}$.
\end{lemma}
\begin{proof} 
  Let $\phi:T'\to T$ by $\phi(\lambda_i)=\mu_{i+1}$. Then
  $L'=\phi^{-1}(L\cap k[\mu_2, \dots, \mu_c])$ so
  $T/(L+\mu_1)\cong T/(\phi(L')+\mu_1)$ as $T$-modules. By the mapping
  cone (which is minimal since $\mu_1$ is regular modulo $\phi(L')$),
  $\beta^{(L+\mu_1)}_{i,j}=
  \beta^{\phi(L')+\mu_1}_{i,j}=\beta^{\phi(L')}_{i-1,j-1}+\beta^{\phi(L')}_{i,j}=
  \beta^{L'}_{i-1,j-1}+\beta^{L'}_{i,j}$
  as required for the second equality.

  For the first equality, note that the result is obvious if $j=0,1$,
  so we take $j>1$. Given $I\subseteq T$, let $G(I)_j$ be the degree
  $j$ minimal generators of $I$.  Then it is easy to see that
  $G(L)_j=G(L+\mu_1)_j\sqcup \mu_1G(L:\mu_1)_{j-1}$ (since $j>1$) so
  the theorem holds if $i=1$. The result is also obvious if $i=0$ so
  we suppose $i>1$.

  Now write $G(I)_{j,k}$ to be set of all degree $j$ minimal monomial
  generators $m\in I$ with $\max\{p\ |\ p\in {\rm supp}(m)\}=k$. Since
  $L$, $(L:\mu_1)$, and $L+\mu_1$ are all lex, the Eliahou-Kervaire
  formula (see \cite{HerzogHibi:MonomialIdeals}, equation 7.7 modulo a
  slight typo) implies that for $i>0$
  \begin{eqnarray*}
    \beta_{i,j}^L&=&\sum_{k=1}^c {k-1\choose i-1}\left|G(L)_{j-i+1,k}\right|,\\
    \beta_{i,j-1}^{L:\mu_1}&=&\sum_{k=1}^c {k-1\choose
                               i-1}\left|G(L:\mu_1)_{j-i,k}\right|, \mbox{
                               and}\\
    \beta_{i,j}^{L+\mu_1}&=&\sum_{k=1}^c {k-1\choose
                             i-1}\left|G(L+\mu_1)_{j-i+1,k}\right|.
\end{eqnarray*}

So it is enough to show that whenever $i>1$ and
${k-1\choose i-1}\ne 0$, then
$$G(L)_{j-i+1,k}=\mu_1G(L:\mu_1)_{j-i,k}\sqcup G(L+\mu_1)_{j-i+1,k}, $$ Because
$i>1$, we may assume $k>1$, whence it is obvious that
$ \mu_1G(L:\mu_1)_{j-i,k}\cap G(L+\mu_1)_{j-i+1,k}=\emptyset$.

Now, if $m\in G(L)_{j-i+1,k}$ and $1\not\in {\rm supp}(m)$, then
$m\in G(L+\mu_1)_{j-i+1,k}$. Otherwise, $\frac{m}{\mu_1}\in L:\mu_1$.
Since $m\ne \mu_1$ (because $k>1$),
$k\in {\rm supp}\left(\frac{m}{\mu_1}\right)$ and obviously $m/\mu_1$
is a minimal generator of $L:\mu_1$, so
$m\in \mu_1G(L:\mu_1)_{j-i,k}$ as required.

For the other direction, we note that
$G(L+\mu_1)_{j-i+1,k}\subseteq G(L)_{j-i+1,k}$ (recall that $k>1$). So
suppose that $m\in G(L:\mu_1)_{j-i,k}$. If $m\mu_1$ is not minimal
in $L$, then there is $p\in {\rm supp}\left(m\mu_1\right)$ such that
$\frac{m\mu_1}{\mu_p}\in L$. If $p\ne 1$, then
$\frac{m}{\mu_p}\in L:\mu_1$ a contradiction. But if $p=1$, then
$m\in L$, so $\frac{m\mu_1}{\mu_k}\in L$ (because $L$ is lex) whence
$\frac{m}{\mu_k}\in L:\mu_1$, a contradiction.
\end{proof}

Now computing the Betti diagram of $R/I_{c,t}(\H)$ is simply a matter
of unraveling the induction.

\begin{theorem}\label{t:IBN}  Let $\H$ be an $\mathcal O$-sequence
  with $\ell(\H)<\infty$, $c\ge \H(1)$, $t\ge \ell(\H)$,
  $T=k[\mu_1, \dots, \mu_c]$, and $L\subseteq T$ be the lex ideal
  attaining $\H$. Then $\beta^{I_{c,t}(\H)}_{i,j} = \beta^L_{i,j}$.
\end{theorem}
\begin{proof}
  If $\H=0$, then $I_{c,t}(\H)=R_{c,t,0}$ and $T=L$ and the result
  follows, so we assume $\H\ne 0$ and hence $t>-1$.

  If $c=0$, then $I_{c,t}(\H)=0$ and $\H=1$. It follows that
  $L=(0)\subseteq T=k$ and we are done. 
  If $c=1$, then $I_{1,t}(\H)=(x_0\cdots x_t)$ (see example
  (\ref{e:IctH})) and $L=(\mu_1^{\ell(\H)+1})$ so the result is
  obvious.

  Now suppose that $c>0$, $t>-1$, and $\H\ne 0$. 
  Since $H(T/L:\mu_1)=\H^c$, by induction we have
  $\beta^{I_{c,t-1}(\H^c)}_{i,j}=\beta^{L:\mu_1}_{i,j}$, and thus
  $\beta^{L:\mu_1}_{i,j-1} =
  \beta^{\omega_{c,t}I_{c,t-1}(\H^c)}_{i,j}$
  (by Lemma (\ref{l:multbyomega})). Similarly, if we let $L'$ be the
  lex ideal in the polynomial ring with $c-1$ variables and Hilbert
  function $\H_c=H(T/L+\mu_1)$, then
  $\beta^{L'}_{i,j}=\beta^{I_{c-1,t}(\H_c)}_{i,j}$ so
  $\beta^{L'}_{i,j-1}=\beta^{\omega_{c,t}I_{c-1,t}(\H_c)}_{i,j}$.


  Obviously $\beta^{I_{c,t}(\H)}_{0,1}=0=\beta^L_{0,1}$ and
  $\beta^{I_{c,t}(\H)}_{1,1}=c-\H(1)=\beta^L_{1,1}$ (by Proposition
  (\ref{p:IctHlinears}) and because $H(T/L)=\H$).

  By Proposition (\ref{p:intersection3}), the sequence
   $$0\to \frac{R}{\omega_{c,t}I_{c-1,t}(\H_c)} \to
   \frac{R}{I_{c-1,t}(\H_c)} \oplus
   \frac{R}{\omega_{c,t}I_{c,t-1}(\H^c)} \to
   \frac{R}{I_{c,t}(\H)}\to 0$$
   is exact, and we claim that the mapping cone resolution of
   $R/I_{c,t}(\H)$ is minimal except in degree $1$ between the zeroth
   and first step. This completes the proof since then, for
   $(0,1)\ne (i,j)\ne (1,1)$,
   $\beta^{I{c,t}(\H)}_{i,j}=
   \beta^{L'}_{i-1,j-1}+\beta^{L'}_{i,j}+\beta^{L:\mu_1}_{i,j-1}=
   \beta^L_{i,j}$
   as required (the last equality from Lemma
   (\ref{l:lexbettinumbers})).

   Consider first the chain map induced by
   $R/\omega_{c,t}I_{c-1,t}(\H_c)\to R/I_{c-1,t}(\H_c)$. Note that
   every generator of $\omega_{c,t}I_{c-1,t}(\H_c)$ is divisible by
   $\omega_{c,t}$. Thus, the multidegree of each minimal generator at
   each step of a minimal free resolution of
   $R/\omega_{c,t}I_{c-1,t}(\H_c)$ must be positive with respect to
   $\omega_{c,t}$. Since $I_{c-1,t}(\H_c)\subseteq R_{c-1,t}$,
   however, the multidegree of each minimal generator at each step of
   a minimal free resolution of $R/I_{c-1,t}(\H_c)$ must be zero with
   respect to $\omega_{c,t}$. It follows that the chain map induced by
   tensoring by $k$ must be zero.

   Now consider the chain map induced by
   $R/\omega_{c,t}I_{c-1,t}(\H_c)\xrightarrow{\delta}
   R/\omega_{c,t}I_{c,t-1}(\H^c).$
   Obviously if $\overline{1}$ is the generator of
   $R/\omega_{c,t}I_{c-1,t}(\H_c)$, then $\delta_0(\overline{1})$ is
   the generator of $R/\omega_{c,t}I_{c,t-1}(\H^c),$ so the mapping
   cone resolution is not minimal (between the zeroth and first step
   in degree $1$).  For $j>1$, however, no cancellation can
   occur. Indeed, let $\alpha$ be the minimal exponent such that
   $\mu_2^\alpha\in L$ (we use here that $c>1$).  Then
   $\beta^{L'}_{i,j}=0$ for $1<j<\alpha+i-1$, so
   $\beta^{\omega_{c,t}I_{c-1,t}(\H_c)}_{i,j}=0$ for $2<j<\alpha+i$
   and it is enough to show that
   $\beta^{\omega_{c,t}I_{c,t-1}(\H^c)}_{i,j}=0$ for $j\ge \alpha+i$,
   i.e., that $\beta^{L:\mu_1}_{i,j}=0$ for $j\ge \alpha+i-1$. So we
   need to show that the regularity of $T/L:\mu_1$ is $< \alpha-1$.
   Now $\mu_2^\alpha\in L$, so $\mu_1\mu_c^{\alpha-1}\in L$, and hence
   $\mu_c^{\alpha-1}\in L:\mu_1$. Since it is well know that the
   regularity of the quotient of a lex ideal is equal to one minus the
   minimal power of $\mu_c$ the ideal contains, we have that the
   regularity of the quotient of $L:\mu_1$ is $\le \alpha-2$ as
   required.

   Thus for $j>1$ the degrees of the generators at each step in
   minimal resolutions of $R/\omega_{c,t}I_{c-1,t}(\H_c)$ and
   $R/\omega_{c,t}I_{c,t-1}(\H^c)$ never coincide, hence we conclude
   that no cancellation can occur (for $j>1$), and thus the mapping
   cone is minimal except in degree $1$ between the zeroth and first
   step as required.
\end{proof}

\begin{corollary}\label{c:IctHHilbertfunction}
  Let $\H$ be an $\mathcal O$-sequence with $\ell(\H)<\infty$,
  $c\ge \H(1)$, and $t\ge \ell(\H)$. Then
  $\Delta^{2t+s}H(R_{c,t,s}/I_{c,t}(\H)) = \H$.
\end{corollary}
\begin{proof}
  This is trivial if $\H=0$. If $\H\ne 0$, then use Theorem
  (\ref{t:IisCM}). Since $k$ is infinite an Artinian reduction of
  $R/I_{c,t}(\H)$ and $T/L$ have the same graded Betti numbers and
  hence the same Hilbert function.
\end{proof}

\section{Sub-ideals of $I_{c,t}(\H)$}\label{s:subideals}

The goal of this section is to construct a Gorenstein ideal inside of
$I_{c,t}(\H)$ with which to form the link.

The first step is to identify two sub-ideals of $I_{c,t}(\H)$ and
determining how they relate to one another. We also record a few facts
about these new families which will prove useful when we consider the
weak Lefschetz property.

\begin{definition}
  Let $c\ge 0$, $t\ge -1$, and ${\mathfrak m}$ be the unique
  homogeneous maximal ideal in $T=k[\mu_1, \dots, \mu_{c}]$. Then we
  write $\mathbb H_{c,t} = H(T/{\mathfrak m}^{t+1})$.
\end{definition}

\begin{remark}
  The introduction of a doubly subscripted Hilbert function could turn
  following our iterative definition into an unmitigated disaster. We
  avoid this difficultly by introducing special notation for
  $I_{c,t}(\mathbb H_{c,t})$ which turns out to respect our inductive
  construction.
\end{remark}

\begin{definition}
  Given $c\ge 0$ and $t\ge -1$, then write
  $A_{c,t}=I_{c,t}(\mathbb H_{c,t})$.
\end{definition}

\begin{remark}\label{r:Act}
  Given and $\mathcal O$-sequence $\H$ with $-1<\ell(\H)<\infty$,
  $c\ge \H(1)$, and $t\ge \ell(\H)$, then $\H\le \mathbb H_{c,t}$ and
  hence $A_{c,t}\subseteq I_{c,t}(\H)$.  Furthermore, it is easy to
  see that $({\mathbb H}_{c,t})_c={\mathbb H}_{c-1,t}$ and
  $({\mathbb H}_{c,t})^c={\mathbb H}_{c,t-1}$, so the $A_{c,t}$ follow
  the same iterative rule as the $I_{c,t,}(\H)$. That is,
  $$A_{c,t}=A_{c-1,t}R_{c,t}+\omega_{c,t}A_{c,t-1}R_{c,t}$$ with
  $A_{c,-1}=R_{c,-1}$ and for $t>-1$, $A_{0,t}=(0)\subseteq
  R_{0,t}$. 
  Since the lex ideal attaining $\mathbb H_{c,t}$ is
  $(\mathfrak m)^{t+1}$ it is immediate from Theorem (\ref{t:IBN})
  that $A_{c,t}$ is generated in degree $t+1$ (when
  nonzero).
\end{remark}

We also need an ideal in $I_{c,t}(\H)$ for which $x_0$ is a non-zero-divisor.

\begin{definition} Let $c\ge 1$ and $t\ge -1$. Then we write
  $A'_{c,t}$ to be the ideal in $R_{c,t}$ obtained by removing all
  minimal generators of $A_{c,t}$ which are divisible by $x_0$ (and
  take the ideal generated by no elements to be zero).
\end{definition}

\begin{example}\label{e:Ap2t}
  For example, 
\begin{eqnarray*} 
  A'_{3,2}& =& (y_0 y_1 y_2, x_3 y_0   y_1,x_2 x_3 y_0,x_1 x_2 x_3)\\
          &=& (y_0 y_1 y_2 +x_3(y_0 y_1,x_2 y_0,x_1 x_2))=
              A'_{2,2}+x_3A'_{3,1}
\end{eqnarray*}
By induction it is easy to show that $A'_{2,t}= (y_0\cdots y_t)$ for
all $t\ge -1$ (where the empty product is taken to be $1$ by
convention).
%
\end{example}

\begin{remark}\label{r:Aprime}
  It is immediate that the conclusions of Propositions
  (\ref{p:sqfree}), (\ref{p:HFcomparison}), and
  (\ref{p:subsetcIversion}) hold for $A'_{c,t}$ and additionally
  $A'_{c,t}$ is generated in degree $t+1$ (when nonzero).  Following
  Remark (\ref{r:Act}), it is also easy to see that
  $A'_{c,t}=A'_{c-1,t}+\omega_{c,t}A'_{c,t-1}$ with initial values
  $A'_{c,-1}=R_{c,-1}=A_{c,-1}$ and $A'_{1,t}=(0)$ for $t>-1$, so that
  the conclusion of Proposition (\ref{p:intersection3}) holds.
  Although the $A'_{c,t}$ versions of Proposition (\ref{p:minprimes})
  and Theorem (\ref{t:IisCM}) must be modified slightly (as below),
  the proofs are nearly identical---we change base cases as well as
  the statement of the induction hypothesis in each step---and hence
  are omitted.
\end{remark}

\begin{proposition}\label{p:Aprimeminprimes}
  Let $c\ge 1$ and $t\ge 0$. Then any minimal prime of $A'_{c,t}$ is
  generated by $c-1$ variables, exactly $\lfloor\frac{c}{2}\rfloor$ of
  which are from $\{y_0, \dots, y_{t+\lfloor\frac{c-2}{2}\rfloor}\}$.
\end{proposition}



\begin{theorem}\label{t:Aprimect}
  Let $c\ge 1$ and $t\ge 0$. Then $R_{c,t,s}/A'_{c,t}$ is a dimension
  $2t+s+1$ Cohen-Macaulay algebra with projective dimension $c-1$.
\end{theorem}

We take here the opportunity to explore the relationship between the
$A_{c,t}$ and $A'_{c,t}$.

\begin{proposition}\label{p:disjointminprimes}
  Let $c\ge 2$ and $t\ge 0$. Then $A'_{c,t-1}R_{c,t}$ is contained in
  no minimal prime of $A_{c-1,t}R_{c,t}$.
\end{proposition}
\begin{proof}

  We do induction on $c$ and $t$. If $t=0$, then
  $A'_{c,-1}R_{c,0}=R_{c,0}$ so the result follows from Proposition
  (\ref{p:minprimes}).

  If $c=2$ and $t>0$, then the result follows from Propositions
  (\ref{p:minprimes}) and (\ref{p:Aprimeminprimes}).

  Now suppose that $c>2$ and $t>0$ and $P$ is a minimal prime of
  $A_{c-1,t}$.  As we saw in the proof of Proposition
  (\ref{p:minprimes}) either $P=Q+(\omega_{c-1,t})$ for some
  $Q\subseteq R_{c-2,t}$ minimal over $A_{c-2,t}$ or
  $P\subseteq R_{c-1,t-1}$ is minimal over $A_{c-1,t-1}$.

  If $P$ is minimal over $A_{c-1,t-1}$ then
  $A'_{c,t-2}\not\subseteq P$ by induction, $\omega_{c,t-1}\not \in P$
  since $\omega_{c,t-1}\not\in R_{c-1,t-1}$, and hence
  $\omega_{c,t-1}A'_{c,t-2}\not\subseteq P$. We conclude that
  $A'_{c,t-1}=A'_{c-1,t-1}+\omega_{c,t-1}A'_{c,t-2}\not\subseteq P$ as
  required.

  If $P=Q+(\omega_{c-1,t})$ with $Q$ minimal over $A_{c-2,t}$, then
  $A'_{c-1,t-1}\not\subseteq Q$ by induction on $c$, hence
  $A'_{c-1,t-1}\not\subseteq Q+(\omega_{c-1,t})=P$ since
  $\omega_{c-1,t}\not\in R_{c-1,t-1}$. We conclude that
  $A'_{c,t-1}=A'_{c-1,t-1}+\omega_{c,t-1}A'_{c,t-2}\not\subseteq P$ as
  required.
\end{proof}

Note that $R_{c,t}=R_{c+2,t-1}$ for $c,t \ge 0$, and thus $A_{c,t}$
and $A'_{c+2,t-1}$ are initially defined over the same ring. In fact,
more is true.

\begin{proposition}\label{p:cuptdown}
  Let $c,t\ge 0$. Then $A_{c,t}\subseteq A'_{c+2,t-1}$.
\end{proposition}
\begin{proof}
  We do induction on $c$ and $t$, the $c=0$ and $t=0$ cases being
  obvious.  So suppose $c,t>0$. By induction
  $A_{c-1,t}\subseteq A'_{c+1,t-1}$ and
  $A_{c,t-1}\subseteq A'_{c+2,t-2}$. Since
  $\omega_{c,t}=\omega_{c+2,t-1}$ it follows that
  $A_{c,t}= A_{c-1,t}+ \omega_{c,t}A_{c,t-1} \subseteq A'_{c+1,t-1}+
  \omega_{c+2,t-1}A'_{c+2,t-2} =A'_{c+2,t-1}$ as required.
\end{proof}

These facts can be used to give information about the residual of
$A_{c,t}$ in $A'_{c,t}$.

\begin{proposition}\label{p:AAprimecolon} Let $c\ge 2$ and $t\ge 0$. Then
  $A'_{c,t-1}:A_{c-1,t}=A'_{c,t-1}$.
\end{proposition}
\begin{proof} We take $t>0$ since the $t=0$ case is obvious. Let
  $\lambda$ be homogeneous such that
  $\lambda A_{c-1,t}\subseteq A'_{c,t-1}.$ Since $A'_{c,t-1}$ is a
  squarefree monomial ideal (see Remark (\ref{r:Aprime})), it is the
  intersection of its minimal primes, and thus if $\lambda\in P$ for
  all $P$ minimal over $A'_{c,t-1}$, then $\lambda\in A'_{c,t-1}$ as
  required.

  So we suppose that for each $a\in A_{c-1,t}$ there is a minimal
  prime $P_a$ of $A'_{c,t-1}$ such that $a\in P_a$. It follows that
  $A_{c-1,t}$ is contained in the union of the minimal primes of
  $A'_{c,t-1}$ and hence is contained in one of them, call it $P$, by
  Prime Avoidance.

  By Proposition (\ref{p:Aprimeminprimes}), $P$ is generated by $c-1$
  variables, and by Proposition (\ref{p:minprimes}), the same is true
  for any minimal prime of $A_{c-1,t}$. It follows that $P$ is minimal
  over $A_{c-1,t}$, but this contradicts Proposition
  (\ref{p:disjointminprimes}).
\end{proof}

\begin{proposition} \label{p:colonA} Let $c\ge 2$ and $t\ge -1$. Then
  $A_{c-2,t}:A_{c-1,t}=A_{c-2,t}$.
\end{proposition}
\begin{proof} The proof is obvious if $t=-1$, so suppose $t\ge 0$ and
  let $\lambda$ be homogeneous such that
  $\lambda A_{c-1,t}\subseteq A_{c-2,t}$. Then
  $\lambda A_{c-1,t}\subseteq A_{c-2,t}\subseteq A'_{c,t-1}$ by
  Proposition (\ref{p:cuptdown}) and hence $\lambda\in A'_{c-1,t}$ by
  Proposition (\ref{p:AAprimecolon}). Since
  $A'_{c-1,t}\subseteq A_{c-1,t}$ we have that $\lambda\in A_{c-1,t}$,
  so $\lambda^2\in A_{c-2,t}$, and hence $\lambda\in A_{c-2,t}$ since
  $A_{c-2,t}$ is squarefree (Proposition (\ref{p:sqfree})).
\end{proof}

We can now use the $A_{c,t}$ and $A'_{c,t}$ to form a Gorenstein ideal
inside of $I_{c,t}(\H)$ with which to form the link.

\begin{definition} Let $c\ge 1$ and $t,s\ge 0$. Then we define
  $$G_{1,t,s} = (x_0\cdots x_ty_0\cdots y_{t-1}z_0\cdots
  z_{s-1})\subseteq R_{1,t,s},$$ and for $c\ge 2$,
  $$G_{c,t,s}= A_{c-1,t}R_{c,t,s}+\omega_{c,t} z_0\cdots z_{s-1}
  A'_{c,t-1}R_{c,t,s}.$$ When there can be no confusion, we write
  $\overline{z}$ to denote the product $z_0\cdots z_{s-1}$,
  so $$G_{c,t,s}= A_{c-1,t}+\omega_{c,t}\overline{z} A'_{c,t-1}.$$
\end{definition}

\begin{remark}\label{r:containment} Since
  $\omega_{c,t}\overline{z}A'_{c,t-1}\subseteq A'_{c,t}$,
  $G_{1,t,s}\subseteq A_{1,t}$, and
  $A_{c-1,t},A'_{c,t}\subseteq A_{c,t}\subseteq I_{c,t}(\H)$, (see
  Remark (\ref{r:Act})), it is immediate that
  $G_{c,t,s}\subseteq I_{c,t}(\H)R_{c,t,s}$. Similarly since
  $A_{c-1,t}$ and $\omega_{c,t}\overline{z}A'_{c,t-1}$ are squarefree
  (the latter because $A'_{c-1,t}\subseteq R_{c-1,t,0}$) it follows
  that $G_{c,t,s}$ is a squarefree monomial ideal.
\end{remark}


It is not difficult to show that $R/G_{c,t,s}$ is Gorenstein using
induction. For the sake of the exposition, we first make one
observation in a Lemma.

\begin{lemma}\label{l:Gintersection}
  Let $c\ge 2$ and $t,s\ge 0$. Then
  $$ A_{c-1,t}R_{c,t,s}\cap \omega_{c,t}
  \overline{z}A'_{c,t-1}R_{c,t,s}=\omega_{c,t} \overline{z}
  G_{c-1,t,0}R_{c,t,s}$$
\end{lemma}
\begin{proof}



  Note that $\omega_{c,t}\overline{z}\not\in R_{c-1,t}$,
  so $$A_{c-1,t}R_{c,t,s}\cap \omega_{c,t} \overline{z}
  A'_{c,t-1}R_{c,t,s}=\omega_{c,t}\overline{z}(A_{c-1,t}R_{c,t,s}\cap
  A'_{c,t-1}R_{c,t,s})$$ and thus it is equivalent to show that
  $$A_{c-1,t}R_{c,t}\cap A'_{c,t-1}R_{c,t}= G_{c-1,t,0}R_{c,t}.$$ 
  First we show that
  $A_{c-1,t}\cap A'_{c,t-1}\subseteq G_{c-1,t,0}R_{c,t}$ by induction
  on $t$.

If $t=0$, then 
 $$A_{c-1,0}R_{c,0}\cap
 A'_{c,-1}R_{c,0}= A_{c-1,0}R_{c,0}\cap R_{c,0}$$
 $$= A_{c-2,0}R_{c,0}+\omega_{c-1,0}A_{c-1,-1}R_{c,0}$$
 $$= A_{c-2,0}R_{c,0} +
 \omega_{c-1,0}A'_{c-1,-1}R_{c,0}=G_{c-1,0,0}R_{c,0}$$
 as required.

 Now suppose $t>0$ and we have a monomial
 $m\in A_{c-1,t}\cap A'_{c,t-1}.$ Of course,
 $A_{c-1,t}= A_{c-2,t}+\omega_{c-1,t}A_{c-1,t-1}$ and
 $A'_{c,t-1}= A'_{c-1,t-1}+\omega_{c,t-1}A'_{c,t-2}.$ If
 $m\in A_{c-2,t}$, then we are finished since
 $A_{c-2,t}\subseteq G_{c-1,t,0}$, so we suppose that
 $m\in \omega_{c-1,t}A_{c-1,t-1}$, and thus that
 $\omega_{c-1,t}\ |\ m$. If $m\in A'_{c-1,t-1}$, then
 $m\in \omega_{c-1,t}A'_{c-1,t-1}$ (since
 $\omega_{c-1,t}\not\in R_{c-1,t-1}$) which is enough since
 $\omega_{c-1,t}A'_{c-1,t-1}\subseteq G_{c-1,t,0}$.  Thus we conclude
 that
 $$m\in \omega_{c-1,t}A_{c-1,t-1}\cap
 \omega_{c,t-1}A'_{c,t-2}$$
 $$\subseteq A_{c-1,t-1}\cap A'_{c,t-2}\subseteq
 G_{c-1,t-1,0}$$
 $$= A_{c-2,t-1}+\omega_{c-1,t-1}A'_{c-1,t-2}$$ where the second to
 last inclusion is by induction. Suppose first that
 $m\in A_{c-2,t-1}$. Then $m\in \omega_{c,t-1}A_{c-2,t-1}$ since
 $\omega_{c,t-1}\ |\ m$ but $\omega_{c,t-1}\not\in R_{c-2,t-1}$.
 Noting that $\omega_{c,t-1}= \omega_{c-2,t}$ we have
 $m\in \omega_{c-2,t}A_{c-2,t-1}\subseteq A_{c-2,t}\subseteq
 G_{c-1,t,0}$
 as required. Finally, we suppose that
 $m\in \omega_{c-1,t-1}A'_{c-1,t-2}$. But
 $\omega_{c-1,t-1}A'_{c-1,t-2}\subseteq A'_{c-1,t-1}$, a case 
 already covered.

 The other direction is simpler.  We have
 $G_{c-1,t,0} = A_{c-2,t}+\omega_{c-1,t}A'_{c-1,t-1}.$ But
 $A_{c-2,t}\subseteq A_{c-1,t}, A'_{c,t-1}$, the latter by Proposition
 (\ref{p:cuptdown}), so $A_{c-2,t}\subseteq A_{c-1,t}\cap A'_{c,t-1}.$
 Also,
 $\omega_{c-1,t}A'_{c-1,t-1}\subseteq A'_{c-1,t-1},
 \omega_{c-1,t}A_{c-1,t-1}$
 and $A'_{c-1,t-1}\subseteq A'_{c,t-1}$ while
 $\omega_{c-1,t}A_{c-1,t-1}\subseteq A_{c-1,t}$ so
 $\omega_{c-1,t}A'_{c-1,t-1}\subseteq A_{c-1,t}\cap A'_{c,t-1}$ as
 well, which completes the proof.
\end{proof}

\begin{theorem}\label{t:Ggor}
  Let $c\ge 1$ and $t,s\ge 0$. Then $R/G_{c,t,s}$ is Gorenstein of
  dimension $2t+s$, projective dimension $c$, and, if $T_\bullet$ is a
  minimal free resolution of $R/G_{c,t,s}$, then $T_c$ is a rank 1
  free module generated in degree $2t+s+c$.
\end{theorem}
\begin{proof}
  The $c=1$ case is immediate and $c=2$ follows because (see examples
  (\ref{e:IctH}) and (\ref{e:Ap2t}))
  $G_{2,t,s}=A_{1,t}+\overline{z}y_tA'_{2,t-1} = (x_0\cdots
  x_t,z_0\cdots z_{s-1}y_0\cdots y_t)$.

  So suppose $c\ge 3$. By induction, $R_{c-1,t,0}/G_{c-1,t,0}$ is
  Gorenstein with dimension $2t$, projective dimension $c-1$, and the
  generator of the rank 1 free module at the $(c-1)^{\rm st}$ step of
  a minimal free resolution of $R_{c-1,t,0}/G_{c-1,t,0}$ has degree
  $2t+c-1$.  By Lemma (\ref{l:multbyomega}) and the discussion
  preceding it,
  $R/\omega_{c,t}\overline{z}G_{c-1,t,0}$ has depth
  $2t+s+1$, projective dimension $c-1$, and if $T'_\bullet$ is a
  minimal free resolution of
  $R/\omega_{c,t}\overline{z}G_{c-1,t,0}$, then
  $T'_{c-1}$ is rank 1 and generated in degree $2t+s+c$.

  By Proposition (\ref{l:Gintersection}),
  $\omega_{c,t}\overline{z}G_{c-1,t,0} = A_{c-1,t}\cap
  \omega_{c,t}\overline{z}A'_{c,t-1}$,
  so the long exact sequence in Ext on the short exact sequence
  $$0\to \frac{R}{\omega_{c,t}\overline{z}G_{c-1,t,0}} \to
  \frac{R}{A_{c-1,t}}\oplus\frac{R}{\omega_{c,t}\overline{z}A'_{c,t-1}}
  \to \frac{R}{G_{c,t,s}}\to 0$$
  shows that
  $${\rm dep}\frac{R}{G_{c,t,s}}\ge {\rm min}\left\{{\rm
      dep}\frac{R}{\omega_{c,t} \overline{z}G_{c-1,t,0}}-1, {\rm dep}
    \frac{R}{A_{c-1,t}}, {\rm dep}
    \frac{R}{\omega_{c,t}\overline{z}A'_{c,t-1}}\right\}$$
  $$= \min\{2t+s, 2t+s+1, 2(t-1)+s+2\}=2t+s$$ by
  Lemma (\ref{l:multbyomega}) and the discussion preceding it, as
  well as Theorems (\ref{t:IisCM}) and (\ref{t:Aprimect}). 

  Now no minimal prime of $A_{c-1,t}$ contains
  $\omega_{c,t}\overline{z}A'_{c,t-1}$ (this follows because
  $\omega_{c,t}\overline{z}\not\in R_{c-1,t}$ and by Proposition
  (\ref{p:disjointminprimes})) so
  $\dim R/G_{c,t,s}<\dim R/A_{c-1,t}=2t+s+1$. It follows that
  $R/G_{c,t,s}$ is Cohen-Macaulay of dimension $2t+s$ and projective
  dimension $c$ (the latter fact by the Auslander-Buchsbaum formula).

  So consider the mapping cone resolution obtained from the short
  exact sequence above. Note that the projective dimensions of each of
  $R/\omega_{c,t}\overline{z}G_{c-1,t,0},$ $R/A_{c-1,t}$, and
  $R/\omega_{c,t}\overline{z}A'_{c,t-1}$ is $c-1$ (again see Lemma
  (\ref{l:multbyomega}) and the comments preceding it, as well as
  Theorems (\ref{t:IisCM}) and (\ref{t:Aprimect})), but the projective
  dimension of $R/G_{c,t,s}$ is $c$. It follows by the mapping cone
  construction that the last term in a minimal free resolution of
  $R/G_{c,t,s}$ is a non-zero free submodule of $T'_{c-1}$, a rank 1
  free module generated in degree $2t+s+c$, and we are finished.
\end{proof}

\begin{corollary}\label{c:ellG}
  Let $c\ge 1$ and $t,s\ge 0$. Then
  $\ell(\Delta^{2t+s}H(R/G_{c,t,s})) =2t+s$.
\end{corollary}
\begin{proof}
 Since $R/G_{c,t,s}$ is dimension $2t+s$ Cohen-Macaulay, 
  $\Delta^{2t+s}\H(R/G_{c,t,s})$ is the Hilbert function of an
  Artinian reduction of $R/G_{c,t,s}$, say $S$.  Of course $S$ is
  Gorenstein with the same graded Betti numbers as $R/G_{c,t,s}$ (as
  is well known). But in the dimension zero case,
  $\ell(\Delta^{2t+s}H(R/G_{c,t,s}))=\ell(H(S))$ equals the socle
  degree of $S$ which is seen to be $2t+s$ by Theorem (\ref{t:Ggor}).
\end{proof}

\section{A Gorenstein ideal with Hilbert function
  $\H$}\label{s:gorenstein}

We now have all the pieces required to define a Gorenstein ideal with
Hilbert function $\H$.

\begin{definition}
  Given an SI-sequence $\H$, let $c=H(1)$, $t=\ell(\Delta\H)$, and
  $s=\ell(\H)-2t+1$. If $c=0$, then $t=0$, $s=1$, and we let
  $J_0(\H)=0\subseteq R_{0,0,1}=k[z_0]$. If $c=1$, then $t=0$,
  $s\ge 2$, and we let
  $J_1(\H)=(z_0\cdots z_{s-1})\subseteq k[x_0, z_0, \dots, z_{s-1}]$.

 Finally, for
  $c\ge 2$, we let $J_c(\H)$ be the ideal of $R_{c,t,s}$
  $$J_c(\H)= I_{c-1,t}(\Delta\H)+G_{c-1,t,s}:I_{c-1,t}(\Delta\H).$$
\end{definition}

Recall the classic result of Peskine-Szpiro \cite{PeskineSzpiro}.

\begin{theorem}\label{t:PeskineSzpiro}
  Let $G\subsetneq I$ be ideals of a ring $R$ such that $R/G$ is
  Gorenstein and $R/I$ is Cohen-Macaulay with $\dim(R/I)=d$. Then
  $R/(G:I)$ is Cohen-Macaulay with $\dim(R/(G:I))=d$. Additionally, if
  $I\cap (G:I)=G$ and $I$ and $(G:I)$ share no minimal primes, then
  $R/(I+(G:I))$ is Gorenstein of dimension $d-1$.
\end{theorem}

\begin{theorem}\label{t:JisGor}
  Given an SI-sequence $\H$, $c=H(1)$, $t=\ell(\Delta\H)$ and
  $s=\ell(\H)-2t+1$, $R_{c,t,s}/J_c(\H)$ is a dimension $2t+s$
  Gorenstein $k$-algebra.
\end{theorem}
\begin{proof}
  The result it obvious if $c=0,1$. So let $c\ge 2$. By Theorem
  (\ref{t:IisCM}) and the comments before Lemma (\ref{l:multbyomega}),
  $R/I_{c-1,t}(\Delta\H)$ is Cohen-Macaulay of dimension
  $2t+s+1$. Furthermore, $G_{c-1,t,s}\subseteq I_{c-1,t}(\Delta\H)$ as
  observed in Remark (\ref{r:containment}). So by Theorem
  (\ref{t:PeskineSzpiro}) it is enough to show that
  $I_{c-1,t}(\Delta\H)\cap (G_{c-1,t,s}:I_{c-1,t}(\Delta\H))=
  G_{c-1,t,s}$
  and $I_{c-1,t}(\Delta\H)$ and $G_{c-1,t,s}$ share no minimal primes
  (which shows $G_{c-1,t,s}\subsetneq I_{c-1,t}(\Delta\H)$).  This is
  a subscript-free observation, so we use $I$ and $G$. If $m$ is a
  monomial in $I\cap (G:I)$, then $m^2\in mI\subseteq G$, but $G$ is
  squarefree (Remark (\ref{r:containment})) so $m\in G$ and thus
  $I\cap (G:I)\subseteq G$ as required. The other inclusion is
  obvious.

  Now suppose that $I$ and $(G:I)$ share a minimal prime. Then there
  is a $P$ prime and $x\not\in I$, $y\not\in (G:I)$ such that
  $P=(I:x)=((G:I):y)$. Of course $x\not\in P$, else
  $x^2\in xP\subseteq I \implies x\in I$ since $I$ is squarefree, a
  contradiction. But $xyP^2=xPyP\subseteq I(G:I)\subseteq G$ and hence
  $xyP\subseteq G$ since $G$ is squarefree. Thus
  $xyI\subseteq xyP\subseteq G$ implies $x\in ((G:I):y)= P$, a
  contradiction.

\end{proof}

In order to determine the Hilbert function and graded Betti numbers of
$J_c(\H)$ we need the following Theorem from Davis and Geramita
\cite{DavisGeramita}.

\begin{theorem}[3 in \cite{DavisGeramita}]\label{t:DavisGer}
  Let $G\subseteq R$ be a Gorenstein ideal of dimension $d$ and
  $G\subseteq I$ be such that $R/I$ is dimension $d$ and
  Cohen-Macaulay. Then
  $$\ell(\Delta^d H(R/G)) = \ell(\Delta^d H(R/I))+\min\{\lambda\ |\
  G_\lambda\ne (G:I)_\lambda\}.$$
\end{theorem}

We will make use of Davis and Geramita's Theorem as follows.


\begin{corollary}\label{c:IandJ}
  Let $\H$ be an SI-sequence with $\H(1)\ge 1$, $c=\H(1)$,
  $t=\ell(\Delta\H),$ and $s = \ell(\H)-2t+1$.  Then
  $(J_c(\H))_\lambda = (I_{c-1,t}(\Delta\H))_\lambda$ for
  $\lambda<t+s$.
\end{corollary}
\begin{proof}
  The result is obvious if $c=1$. If $c=2$ then
  $I_{1,t}(\Delta\H)=(x_0\cdots x_t)$ (see example (\ref{e:IctH})),
  $J_2(\H) = (x_0\cdots x_t, y_0\cdots y_{t-1}z_0\cdots z_{s-1})$, and
  the result follows.

  So let $c>2$ and write $d=2t+s$.  Then
  $\ell(\Delta^d H(R_{c-1,t,s}/G_{c-1,t,s}))=2t+s$ and
  $\ell(\Delta^{d}H(R_{c-1,t,s}/I_{c-1,t}(\Delta\H))=t$ by Corollaries
  (\ref{c:ellG}) and (\ref{c:IctHHilbertfunction}), whence by Theorem
  (\ref{t:DavisGer})
  $$\min\{\lambda\ |\ (G_{c-1,t,s})_\lambda\ne
  (G_{c-1,t,s}:I_{c-1,t}(\Delta\H))_\lambda\}=t+s.$$
  We conclude that
  $$(G_{c-1,t,s}:I_{c-1,t}(\Delta\H))_\lambda=(G_{c-1,t,s})_\lambda=
  (A_{c-1,t}+\omega_{c,t}\overline{z}A'_{c,t-1})_\lambda=
  (A_{c-1,t})_\lambda$$ for $\lambda <t+s.$  So
  $$(J_c(\H))_\lambda = (I_{c-1,t}(\Delta\H))_\lambda +
  (G_{c-1,t,s}:I_{c-1,t}(\Delta\H))_\lambda$$
  $$ = (I_{c-1,t}(\Delta\H))_\lambda + (A_{c-1,t})_\lambda =
  (I_{c-1,t}(\Delta\H))_\lambda$$
  for $\lambda<t+s$ by Remark (\ref{r:Act}).
\end{proof}

We can now show that an Artinian reduction of $R/J_c(\H)$ has Hilbert
function $\H$.

\begin{theorem}\label{t:JHF}
  Let $\H$ be an SI-sequence, $c=\H(1)$, $t=\ell(\Delta\H),$ and
  $s=\ell(\H)-2t+1$. Then $\Delta^{2t+s}H(R/J_c(\H)) = \H$.
\end{theorem}
\begin{proof}
  The result is obvious if $c=0$ or $1$, so suppose that $c\ge 2$.

  Let $d=2t+s=\dim R/J_{c}(\H)$ (Theorem (\ref{t:JisGor})). Then each
  of $R/I_{c-1,t}(\Delta\H)$, $R/G_{c-1,t,s}$, and
  $R/(G_{c-1,t,s}:I_{c-1,t}(\Delta\H))$ are Cohen-Macaulay of
  dimension $d+1$ (Theorems (\ref{t:IisCM}), (\ref{t:Ggor}), and
  (\ref{t:PeskineSzpiro}) and the comments before Lemma
  (\ref{l:multbyomega})). Since
  $\ell(\Delta^{d+1}H(R/I_{c-1,t}(\Delta\H)))=t$ (Corollary
  (\ref{c:IctHHilbertfunction})), and
  $$\ell(\Delta^{d+1}H(R/(G_{c-1,t,s}:I_{c-1,t}(\Delta\H))))\le
  \ell(\Delta^{d+1}H(R/G_{c-1,t,s}))=2t+s$$
  (Corollary (\ref{c:ellG})),
  $\Delta^{d}H(R/I_{c-1,t}(\Delta\H),\lambda))$,
  $\Delta^{d}H(R/(G_{c-1,t,s}:I_{c-1,t}(\Delta\H)),\lambda)$, and
  $\Delta^{d}H(R/G_{c-1,t,s},\lambda)$ are constant for
  $\lambda\ge 2t+s$.

  For the moment we write $I$, $G$, and $J$ for the sake of
  readability.  Due to the fact that $R/J$ is Gorenstein of dimension
  $d$, it is sufficient to show that
  $\ell(\Delta^dH(R/J))\le \ell(\H)$ and
  $\Delta^d H(R/J, \lambda)=\H(\lambda)$ for
  $\lambda\le \lfloor \frac{\ell(\H)}{2}\rfloor$ (symmetry takes care
  of the rest).

 The first fact follows from the usual short exact sequence
 $$0\to \frac{R}{G} \to
 \frac{R}{I}\oplus\frac{R}{G:I} \to \frac{R}{J}\to 0$$
 (recall that $I\cap (G:I)=G$, as we saw in Theorem (\ref{t:JisGor})).
 Obviously
 $$\Delta^d H(R/J,\lambda)= \Delta^d
 H(R/I,\lambda) +\Delta^d H(R/G:I,\lambda) - \Delta^d
 H(R/G,\lambda),$$
 and we've seen that each of $\Delta^d H(R/G,\lambda)$,
 $\Delta^d H(R/I,\lambda)$, and $\Delta^d H(R/G:I,\lambda)$ is
 constant for $\lambda \ge 2t+s$. We conclude that
 $\Delta^d H(R/J, \lambda)=0$ in those degrees and thus
 $\ell(\Delta^d H(R/J))\le 2t+s-1= \ell(\H)$ as required.

 We proved in Corollary (\ref{c:IandJ}) that $J_\lambda = I_\lambda$
 for $\lambda< t+s$ and
 $\lfloor\frac{\ell(\H)}{2}\rfloor = \lfloor\frac{2t+s-1}{2}\rfloor
 <t+s$,
 so it is enough to show that $\H(\lambda)= \Delta^d H(R/I,\lambda)$
 for $\lambda \le \lfloor\frac{\ell(\H)}{2}\rfloor$.  Because
 $\Delta^{d+1}H(R/I)=\Delta \H$ (Corollary
 (\ref{c:IctHHilbertfunction})), $\Delta^dH(R/I,\lambda)=\H(\lambda)$
 for $\lambda\le t$.  But $\Delta^d H(R/I,\lambda)= \Delta^dH(R/I,t)$
 for $\lambda \ge t$ (as already observed) and $\H(\lambda)=\H(t)$ for
 $t\le \lambda \le \lfloor\frac{\ell(\H)}{2}\rfloor$ since $\H$ is an
 SI-sequence, which completes the proof.
\end{proof}

Recall that our goal is to show that the graded Betti numbers of
$J_c(\H)$ (for $\ell(\H)>0$) match Migliore and Nagel's upper bound
described in Section (\ref{s:introduction}). In our notation, we wish
to show that for $\H$ an SI-sequence with $\H(1)\ge 1$, $c=\H(1)$,
$t=\ell(\Delta\H)$, $s=\ell(\H)-2t+1$, and $L$ the lex ideal in $c-1$
variables with Hilbert function $\Delta \H$, then
 $$\beta^{J_c(\H)}_{i,i+j}=
 \begin{cases} \beta^L_{i,i+j} & \mbox{if }j\le t+s-2\\
   \beta^L_{i,i+j} +\beta^L_{c-i,c-i+2t+s-1-j} & \mbox{if } t+s-1\le j
   \le t\\ \beta^L_{c-i,c-i+2t+s-1-j} & \mbox{if } j \ge
   t+1.\\ \end{cases}$$
 We note immediately that if $s=1$ then we require
 $$\beta^{J_c(\H)}_{i,i+j}=
 \begin{cases} \beta^L_{i,i+j} & \mbox{if }j\le t-1\\
   \beta^L_{i,i+t} +\beta^L_{c-i,c-i+t} & \mbox{if } j=t \\
   \beta^L_{c-i,c-i+2t-j} & \mbox{if } j \ge t+1. \\ \end{cases}$$

 If $s\ge 2$ then $t+s-1\le j \le t$ is an empty condition.  Moreover,
 in the case $s=2$ we have $\lfloor\frac{2t+s-1}{2}\rfloor = t+s-2$
 while $s>2$ implies that $\beta^L_{i,i+j} = 0$ for
 $t+1\le \lfloor\frac{2t+s-1}{2}\rfloor\le j\le t+s-2$ and
 $\beta^L_{c-i,c-i+2t+s-1-j} = 0$ for
 $t+1\le j \le \lfloor\frac{2t+s-1}{2}\rfloor $ (since the regularity
 of $k[\mu_1, \dots , \mu_{c-1}]/L$ is $t$ because the Hilbert function of
 that quotient is $\Delta\H$).  Thus for $s\ge 2$ we require
 $$\beta^{J_c(\H)}_{i,i+j}=
 \begin{cases} \beta^L_{i,i+j} & \mbox{if }j\le \lfloor\frac{2t+s-1}{2}\rfloor\\
   \beta^L_{c-i,c-i+2t+s-1-j} & \mbox{if } j
   \ge\lfloor\frac{2t+s-1}{2}\rfloor+1.\\ \end{cases}$$

 This leads to a nice interpretation. The Betti diagram we are aiming
 for consists of two copies of a lex Betti diagram, the second rotated
 and shifted, then added to the first. If $s=1$, the shift is such
 that the last row of the first diagram and the first row of the
 shifted diagram coincide (and are added, so the middle row of the
 resulting diagram is the sum of the the last nonzero row of the
 original Betti diagram with its mirror image). If $s\ge 2$ then the
 first nonzero row of the shifted diagram sits $s-1$ rows below the
 last nonzero row of the original (so for $s\ge 3$ there are $s-2$
 rows of zeros between them).

\begin{example}\label{e:JBN}
  Consider for example the SI-sequences
\begin{eqnarray*}
  \H&=&\{1,4,6,4,1\}\ (s=1),\\
  {\mathcal G}&=&\{1,4,6,6,4,1\}\  (s=2), \mbox{ and}\\
  {\mathcal K}&=&\{1,4,6,6,6,4,1\}\  (s=3).\\
\end{eqnarray*}
Since the first difference of each of $\H$, ${\mathcal G}$, and
${\mathcal K}$ is $\{1,3,2\}$, by Theorem (\ref{t:IBN})
$I_{3,2}(\Delta\H)$, $I_{3,2}(\Delta{\mathcal G})$, and
$I_{3,2}({\Delta\mathcal K})$, have the same Betti diagram as the lex
ideal
$L=(\mu_1^2,\mu_1\mu_2,\mu_1\mu_3, \mu_2^2, \mu_2\mu_3^2,\mu_3^3)$,
namely \newline

\begin{tabular}{r|rrrr}
$\beta^L$  &    $1$& $6$& $6$& $3$\\
  \hline
  $0$& $1$& .& .& .\\
  $1$& .& $4$& $4$& $1$\\
  $2$& .& $2$& $4$& $2$\\
\end{tabular}
\newline

The Betti diagrams for $J_4(\H)$, $J_4({\mathcal G})$, and
$J_4({\mathcal K})$ are \newline

\begin{tabular}{r|rrrrr}
  $\beta^{J_4(\H)}$&1&9&16&9&1\\
  \hline
  $0$& $1$& .&  .& .& .\\
  $1$& .& $4$&  $4$& $1$& .\\
  $2$& .& $4$&  $8$& $4$& .\\
  $3$& .& $1$&  $4$& $4$& .\\
  $4$& .& .&  .& .& $1$\\
\end{tabular}
\newline\newline

\begin{tabular}{r|rrrrr}
  $\beta^{J_4({\mathcal G})}$&1&9&16&9&1\\
  \hline
  $0$& $1$& .&  .& .& .\\
  $1$& .& $4$&  $4$& $1$& .\\
  $2$& .& $2$&  $4$& $2$& .\\
  $3$& .& $2$&  $4$& $2$& .\\
  $4$& .& $1$&  $4$& $4$& .\\
  $5$& .& .&  .& .& $1$\\
\end{tabular}
\newline

and
\newline

\begin{tabular}{r|rrrrr}
  $\beta^{J_4({\mathcal K})}$&1&9&16&9&1\\
  \hline
  $0$& $1$& .&  .& .& .\\
  $1$& .& $4$&  $4$& $1$& .\\
  $2$& .& $2$&  $4$& $2$& .\\
  $3$& .& .&  .& .& .\\
  $4$& .& $2$&  $4$& $2$& .\\
  $5$& .& $1$&  $4$& $4$& .\\
  $6$& .& .&  .& .& $1$\\
\end{tabular}

\end{example}

 To prove that the graded Betti numbers of $J_c(\H)$ behave as
 desired, we could simply apply Migliore and Nagel's Corollary 8.2
 \cite{MiglioreNagel:Gorensteins}. Instead we give a different proof.

\begin{theorem}\label{t:JcHBN}
  Let $\H$ be an SI-sequence with $\H(1)\ge 1$, $c= \H(1)$,
  $t=\ell(\Delta\H),$  $s=\ell(\H)-2t+1$, and let $L$ be the lex
  ideal in $k[\mu_1, \dots, \mu_{c-1}]$ with Hilbert function
  $\Delta \H$. Then
 $$\beta_{i,i+j}^{J_{c}(\H)} = \begin{cases}
   \beta^L_{i,i+j} &\mbox{if }j\le t+s-2\\
   \beta^L_{i,i+j}+\beta^L_{c-i,c-i+2t+s-1-j} &\mbox{if }t+s-1\le j\le
   t\\
   \beta^L_{c-i,c-i+2t+s-1-j} &\mbox{if }j \ge t+1.\end{cases}$$
\end{theorem}
\begin{proof}
  Note that by symmetry, it is enough to show the equality for
  $j\le \lfloor\frac{2t+s-1}{2}\rfloor$. To reduce clutter we will
  write $J$ for $J_c(\H)$ and $I$ for $I_{c-1,t}(\Delta\H)$.

  Now, as per our discussion before example (\ref{e:JBN}), if $s\ge 2$
  we must show that $\beta^{J}_{i,i+j}= \beta^{L}_{i,i+j}$ for
  $j\le \lfloor\frac{2t+s-1}{2}\rfloor$, that is, that the Betti
  diagrams of $J$ and $L$ coincide in rows $0$ through
  $\lfloor\frac{2t+s-1}{2}\rfloor$. But this is immediate because
  $\lfloor\frac{2t+s-1}{2}\rfloor\le t+s-1$ and
  $(I)_\lambda= (J)_\lambda$ for $\lambda\le t+s-1$ (by Corollary
  (\ref{c:IandJ})) while the graded Betti numbers of $I$ and $L$
  coincide (by Theorem (\ref{t:IBN})).

  If $s=1$, then Corollaries (\ref{c:IctHHilbertfunction}) and
  (\ref{c:IandJ}) are only sufficient to show that
  $\beta^{J}_{i,i+j}= \beta^{L}_{i,i+j}$ for $j\le t-1$ and
  $\beta^{J}_{i,i+j}= \beta^{L}_{c-i,c-i+2t-j}$ for $j\ge t+1$ (the
  latter by symmetry). Thus it remains to show that
  $\beta^{J}_{i,i+t}=\beta^L_{i,i+t} +\beta^{L}_{c-i,c-i+t}$ for
  $i=0, \dots , c$.

  To accomplish this recall (Theorem 11.3 in \cite{Stanley}) that for
  any homogeneous $M\subseteq S=k[\mu_1, \dots, \mu_c]$,
  $$(1-T)^c\sum_{d\in \N}H(S/M,d)T^d =\sum_{d=0}^c\sum_{j\in
    \N}(-1)^d\beta^{M}_{d,j}T^j .$$
  It follows that fixing a Hilbert function determines the alternating
  sum of the graded Betti numbers along the northeastern heading
  diagonals in the Betti diagram of $S/M$. Moreover, and most relevant
  for us, if the Hilbert function and all but one row of a Betti
  diagram are known, then the unknown entries are computable.
 
  Since $R/J$ is Cohen Macaulay, an Artinian reduction of $R/J$ has
  Hilbert function $\H$ (Theorem (\ref{t:JHF})) with the same graded
  Betti numbers. So
  $$(1-T)^c\sum_{d=0}^{2t}\H(d)T^d=\sum_{d=0}^c\sum_{j\in
    \N}(-1)^d\beta^{J}_{d,j}T^{j},$$
  and hence for fixed $i\in \{0, \dots, c\}$ if $a$ is the coefficient
  of $T^{i+t}$ in $$(1-T)^c\sum_{d=0}^{2t}\H(d)T^d,$$
  then $$a=\sum_{d=0}^c(-1)^d\beta^{J}_{d,i+t}$$ whence
  $$\beta^{J}_{i,i+t} =
  (-1)^ia-\sum_{d=0}^{i-1}(-1)^{d+i}\beta^{J}_{d,i+t} -
  \sum_{d=i+1}^{c}(-1)^{d+i}\beta^{J}_{d,i+t}.$$

  Of course $\beta^J_{d,i+t}=\beta^J_{d,d+(i-d+t)}$. If
  $0\le d\le i-1$, then $i-d+t\ge t+1$, so
  $\beta^J_{d,i+t}=\beta^L_{c-d,c-d+2t-(i-d+t)}=\beta^L_{c-d,c-i+t}$. Similarly,
  if $i+1\le d\le c$, then $i-d+t\le t-1$ so
  $\beta^J_{d,i+t}=\beta^L_{d,i+t}$.

  Thus
  $$\beta^{J}_{i,i+t} =
  (-1)^ia-\sum_{d=0}^{i-1}(-1)^{d+i}\beta^{L}_{c-d,c-i+t} -
  \sum_{d=i+1}^{c}(-1)^{d+i}\beta^{L}_{d,i+t}.$$

  Similarly, we have
  $$(1-T)^{c-1}\sum_{d=0}^{t}\Delta\H(d)T^d= \sum_{d=0}^{c-1}\sum_{j\in
    \N}(-1)^d\beta^L_{d,j}T^j,$$
  and hence, writing $a'$ and $a''$ to be the coefficients of
  $T^{i+t}$ and $T^{c-i+t}$ in
  $(1-T)^{c-1}\sum_{d=0}^{t}\Delta\H(d)T^d$, we have
  $$a'= \sum_{d=0}^{c-1}(-1)^d\beta^L_{d,i+t}=\sum_{d=i}^c(-1)^d\beta^L_{d,i+t}$$ and
  $$a''=\sum_{d=0}^{c-1}(-1)^d\beta^L_{d,c-i+t}=\sum_{d=c-i}^{c}(-1)^d\beta^L_{d,c-i+t}=
  \sum_{d=0}^{i}(-1)^{c-d}\beta^L_{c-d,c-i+t}.$$
  We used that $\beta^L_{d,i+t}=0$ for $d<i$ and $\beta^L_{d,c-i+t}=0$
  for $d<c-i$ because the regularity of $k[\mu_1, \dots, \mu_{c-1}]/L$
  is $t$ (we can read this from the Hilbert function since $L$ is lex)
  and that $\beta^L_{c,j}=0$ since the projective dimension of
  $k[\mu_1, \dots, \mu_{c-1}]/L$ is $c-1$ (note   $k[\mu_1, \dots,
  \mu_{c-1}]/L$ is dimension $0$). Thus
  $$(-1)^{i+1}a'+\beta^L_{i,i+t}
  =-\sum_{d=i+1}^{c}(-1)^{d+i}\beta^L_{d,i+t}$$ and
  $$(-1)^{c-i+1}a''+\beta^L_{c-i,c-i+t}
  =-\sum_{d=0}^{i-1}(-1)^{2c-i-d}\beta^L_{c-d,c-i+t}
  $$ $$=-\sum_{d=0}^{i-1}(-1)^{d+i}\beta^L_{c-d,c-i+t}.$$

  We conclude that
  $$\beta^{J}_{i,i+t} =(-1)^ia+
  (-1)^{i+1}a'+\beta^L_{i,i+t} +(-1)^{c-i+1}a''+\beta^L_{c-i,c-i+t},$$
  so it is enough to show that $a=a'+(-1)^{c}a''$.

  Now recalling that $\Delta\H(d)$ is negative for $d>t$,
  $$(1-T)^c\sum_{d=0}^{2t}\H(d)T^d=
  (1-T)^{c-1}\sum_{d=0}^{2t+1}\Delta\H(d)T^d $$
  $$= (1-T)^{c-1}\sum_{d=0}^t \Delta\H(d)T^d +
  (1-T)^{c-1}\sum_{d=t+1}^{2t+1} \Delta\H(d)T^d,$$
  so $a=a'+b$ where $b$ is the coefficient of $T^{i+t}$ in
  $(1-T)^{c-1}\sum_{d=t+1}^{2t+1} \Delta\H(d)T^d.$ Thus we are
  finished if $b=(-1)^{c}a''$.

  Of course, $\H$ symmetric and $\ell(\H)=2t$, so for
  $t+1\le d\le 2t$, $\H(d)=\H(2t-d)$, and thus
  $\Delta\H(d)=\H(d)-\H(d-1)=\H(2t-d)-\H(2t-d+1)=-\Delta\H(2t-d+1)$. We
  conclude that
  $$(1-T)^{c-1}\sum_{d=t+1}^{2t+1} \Delta\H(d)T^d= -(1-T)^{c-1}\sum_{d=t+1}^{2t+1}
  \Delta\H(2t-d+1)T^d$$
  $$= -(1-T)^{c-1}\sum_{d=0}^{t} \Delta\H(d)T^{2t-d+1}.$$

  So $b$ is the coefficient of $T^{i+t}$ in
  $-(1-T)^{c-1}\sum_{d=0}^{t} \Delta\H(d)T^{2t-d+1}$. Inverting $T$,
  we see that $b$ is the coefficient of $T^{-i-t}$ in
  $$-(1-1/T)^{c-1}\sum_{d=0}^{t} \Delta\H(d)T^{d-2t-1},$$ whence $b$ is
  the coefficient of $T^{c-i+t}$ in
  $$-T^{2t+c}(1-1/T)^{c-1}\sum_{d=0}^{t}
  \Delta\H(d)T^{d-2t-1}$$
  $$=-(T-1)^{c-1}T^{2t+1}\sum_{d=0}^{t}
  \Delta\H(d)T^{d-2t-1}=(-1)^c(1-T)^{c-1}\sum_{d=0}^{t}
  \Delta\H(d)T^{d}$$ and thus $b=(-1)^{c}a''$ as required.
\end{proof}

Our final task is to reduce $R/J$ to a dimension zero Gorenstein ring
with the right Hilbert function and the weak Lefschetz property. To do
so requires an easy observation.

\begin{lemma}\label{l:killz0}
  Let $\H$ be an SI-sequence with $\ell(\H)>0$, $c=\H(1)$,
  $t=\ell(\Delta\H),$ and $s=\ell(\H)-2t+1$. Then
  $J_c(\H)+(z_0)=I_{c-1,t}(\Delta \H)+(z_0)$.
 \end{lemma}
\begin{proof}
  If $c=1$ or $2$, the result is obvious (again, see example
  (\ref{e:IctH})) so we suppose $c>2$.  It is enough to show that
  $G_{c-1,t,s}:I_{c-1,t}(\Delta \H)\subseteq I_{c-1,t}(\Delta
  \H)+(z_0)$.
  Suppose that $\lambda\in G_{c-1,t,s}:I_{c-1,t}(\Delta\H)$. If there
  is a minimal generator $m\in I_{c-1,t}(\Delta \H)$ such that
  $\lambda m\in\omega_{c-1,t}\overline{z}A'_{c-1,t-1}$, then since
  $z_0$ is a non-zero-divisor on $I_{c-1,t}(\Delta \H)$ it follows
  that $z_0$ divides $\lambda$ and we are finished. Thus we may assume
  that
  $\lambda A_{c-1,t}\subseteq \lambda I_{c-1,t}(\Delta \H)\subseteq
  A_{c-2,t}$,
  whence by Proposition (\ref{p:colonA})
  $\lambda\in A_{c-2,t}\subseteq A_{c-1,t}\subseteq I_{c-1,t}(\Delta
  \H)$ as required.
\end{proof}

\begin{corollary}
  Let $\H$ be an SI-sequence with $\H(1)\ge 1$, $c=\H(1)$,
  $t=\ell(\Delta\H),$ and $s=\ell(\H)-2t+1$. Then $R/(J_c(\H)+(z_0))$
  is Cohen-Macaulay of dimension $2t+s$.
\end{corollary}
\begin{proof}
  The result follows from Lemma (\ref{l:killz0}), Theorem
  (\ref{t:IisCM}), and the remarks before Lemma (\ref{l:multbyomega}).
\end{proof}

We can now prove the main result of the paper. 

\begin{theorem}\label{t:Artreduction}
  Let $\H$ be an SI-sequence. Then there is an Artinian Gorenstein
  $k$-algebra with the weak Lefschetz property and Hilbert function
  $\H$ that has unique maximal graded Betti numbers among all Artinian
  Gorenstein $k$-algebra with the weak Lefschetz property and Hilbert
  function $\H$.
\end{theorem}
\begin{proof} If $\H=1$, then this is obvious, so suppose $\ell(\H)>0$
  and let $c=\H(1)$, $t=\ell(\Delta\H),$ and $s=\ell(\H)-2t+1$. Since
  Since $R/J_c(\H)$ and $R/(J_c(\H)+z_0)$ are Cohen Macaulay of
  dimension $2t+s$ and $k$ is infinite, we can find a length $2t+s$
  sequence of linear forms $\ell_1, \dots, \ell_{2t+s}$ which is
  regular on both $R/J_c(\H)$ and $R/(J_c(\H)+z_0)$. Let $\mathbb L$
  be the ideal generated by the $\ell_i$, $S=R_{c,t,s}/\mathbb L$, and
  $\overline{J_c(\H)}$ be the image of $J_c(\H)$ in $S$. Then
  $S/\overline{ J_c(\H)}$ is an Artinian reduction of $R/J_{c}(\H)$
  and hence Gorenstein (Theorem (\ref{t:JisGor})) with Hilbert
  function $\H$ (Theorem (\ref{t:JHF})). To show that
  $S/\overline{ J_c(\H)}$ has the weak Lefschetz property, it is
  enough to show that
  $H(S/(\overline{ J_c(\H)}+\overline{z_0}))=\Delta\H$. But $z_0$ is
  regular on $R/I_{c-1,t}(\Delta \H)$ and
  $J_c(\H)+(z_0)=I_{c-1,t}(\Delta\H)+(z_0)$ (Lemma (\ref{l:killz0})),
  so $\{z_0, \ell_1, \dots, \ell_{2t+s}\}$ is regular on
  $R/I_{c-1,t}(\Delta \H)$ and thus
  $H(S/\overline{J_c(\H)}+\overline{z_0})= H(R/J_c(\H)+{\mathbb
    L}+z_0)= H(R/(I_{c-1,t}(\Delta \H)+z_0+{\mathbb L}))=
  \Delta^{2t+s+1}H(R/I_{c-1,t}(\Delta\H))=\Delta \H$
  (applying Corollary (\ref{c:IctHHilbertfunction})).

  Finally, since the graded Betti numbers of $S/\overline{J_c(\H)}$
  and $R/J_c(\H)$ coincide, we are finished by Theorem (\ref{t:JcHBN})
  and Migliore and Nagel's bound (Theorem (\ref{t:MiglioreNagel})).
\end{proof}

\bibliography{hfbib}
\bibliographystyle{plain}
\end{document}